\documentclass{article}
\usepackage{amsmath}
\usepackage{amsthm}
\usepackage{amssymb}
\usepackage[dvips]{graphicx}
\usepackage{latexsym}
\usepackage{mathtools}
\usepackage{geometry}
 \renewcommand{\equation}

\newtheorem{prop}{Proposition}[section]
\newtheorem{lem}{Lemma}[section]

\newtheorem{fact}{Fact}[section]
\newtheorem{corresp*}{Correspondence}
\newtheorem{facts*}{Facts}
\newtheorem{claim}{Claim}[section]

\theoremstyle{definition}
\newtheorem{rmk}{Remark}[section]
\newtheorem{defini}{Definition}[section]
\newtheorem{example}{Example}[section]

\title{An example of semiquandle and u-polynomials for flat virtual knots via this semiquandle coloring}
\author{Nozomu Sekino}
\date{}

\begin{document}
\maketitle

\begin{abstract}
Flat virtual links are some variant of links, and semiquandles are counterparts of quandles or biquandles, which axiomize the Reidemeister-like moves. 
In this paper, we give some example of semiquandle and introduce an invariant for flat virtual knots using it. 
We also explain that it relates to the u-polynomials for flat virtual knots. 
\end{abstract}

\section{Introduction}
Knots or links can be handled combinatorially by being regarded as equivalence classes generated by Reidemeister moves of diagrams. 
There are many variants of knots and links and corresponding Reidemeister-like moves. 
In this paper, we focus on one of them, called {\it flat virtual knots}. 
Flat virtual knots are introduced and studied in \cite{kauffman} and \cite{upolynomial}. In \cite{upolynomial}, flat virtual knots are called {\it virtual strings}. 
Flat virtual links can be regarded as the resultants of discarding the information of over/under crossings of virtual links. The crossings removed the information are called {\it flat crossings}. 
Following \cite{kauffman} and \cite{upolynomial}, there are many works studying flat virtual knots or links, \cite{chen1}, \cite{chen2}, \cite{gibson1}, \cite{gibson2}, \cite{henrich}, \cite{hrencecin}, \cite{im}, \cite{kauffman2}, \cite{lee}, and so on. 
\cite{flatknotinfo} gives many examples of flat virtual knots and computations of invariants. 
In this paper, we investigate flat virtual knots in terms of {\it semiquandle colorings}. 
In \cite{henrich}, the {semiquandle} structures on sets are defined. 
They are axiomization of Reidemeister-like moves of flat virtual links version, like quandle or biquandle structures for classical or virtual links. 
We will give an example of semiquandle and define some invariant for flat virtual knots in terms of coloring equations of this semiquandle. 
We also check that we can get the {u-polynomial}, which is a classical invariant of flat virtual knots from this invariant. 

The rest of this paper is constructed as follows: 
In Section 2, we define a not necessarily commutative ring $\mathcal{S}_{n}$ for non-negative integer $n$. 
In Section 3, we recall the definition of semiquandles and give $\mathcal{S}_n$ a semiquandle structure. 
In Section 4, we recall the definition of flat virtual links and the useful presentations of flat virtual knots, called {\it Gauss diagrams}. Also, we recall the definition of u-polynomials in terms of Gauss diagrams. 
In Section 5, we recall the definition of a tool we needed for defining our invariant, called {\it quasideterminant}. 
This is defined on matrices with entries of not necessarily commutative ring. 
In Section 6, we construct our invariant and prove the invariance of the choices in the construction, and we explain the relation to the u-polynomials.

\section{The ring $\mathcal{S}_{n}$}
Set $R$ to be $\mathbb{Z}\left< s \right>\oplus \mathbb{Z}\left< \overline{s} \right> \oplus \mathbb{Z} \left<t\right>$, a free $\mathbb{Z}$-module with a basis $\{s, \overline{s}, t\}$, and $T(R)$ to be the tensor algebra over $R$, where tensor operations are considered with $\mathbb{Z}$-coefficients. 
Let $n$ be a non-negative integer. 
Let $I_{n}$ be a two-sided ideal of $T(R)$ generated by elements: 
\begin{itemize}
\item $s \otimes \overline{s} -1$, 
\item $\overline{s} \otimes s -1$, 
\item $t\otimes s-s\otimes \left( {\displaystyle \sum^{n}_{i=1} } t^{\otimes i}\right)$, 
\item $t\otimes \overline{s}-\overline{s}\otimes \left( {\displaystyle \sum^{n}_{i=1} }(-1)^{i-1} t^{\otimes i}\right)$, and 
\item $t^{\otimes n+1}$.
\end{itemize}
Then define ${\mathcal{S}}_{n}$ as a ring $T(R)/ I_n$. 
Abusing notation, the images of $s, \overline{s}, t\in T(R)$ are denoted by $s, s^{-1}, t\in {\mathcal{S}}_{n}$, respectively. 
Informally, ${\mathcal{S}}_{n}$ is a ring of polynomials of $\mathbb{Z}$-coefficients over two non-commutative variables, one of which $s$ is invertible and the other $t$ is not, and satisfying $ts=s\left( {\displaystyle \sum^{n}_{i=1} } t^{ i}\right)$, $ts^{-1}=s^{-1}\left( {\displaystyle \sum^{n}_{i=1} }(-1)^{i-1} t^{i}\right)$ and $t^{n+1}=0$. 
We frequently regard ${\mathcal{S}}_{n}$ in this way and refer ``monomials'' for example. 
We will research properties of $\mathcal{S}_n$ in next subsections.

\begin{rmk}\label{degree}
For a monomial of ${\mathcal{S}}_n$, we call the number of the total exponents of $t$ in this monomial the {\it degree} of it. 
By the relations of ${\mathcal{S}}_n$, a monomial is transformed into some sum of monomials. 
Note that the degree of each monomial in this sum is not less than that of the original monomial. 
Especially, we see that a monomial of degree grater than $n$ must be $0$ in ${\mathcal{S}}_{n}$. 
\end{rmk} 

\subsection{A basis of ${\mathcal{S}}_{n}$ as a $\mathbb{Z}$-module}
By forgetting the product structure, $\mathcal{S}_{n}$ can be regarded as a $\mathbb{Z}$-module. 
\begin{prop} \label{basis}
$\mathcal{S}_{n}$ has a basis $\{ s^{l}t^{m} \mid l, m\in \mathbb{Z}, \ 0\leq m \leq n\}$ as a $\mathbb{Z}$-module.
\end{prop}
\begin{proof}
(sketch)\\
By the relations $ts=s\left( {\displaystyle \sum^{n}_{i=1} } t^{ i}\right)$, $ts^{-1}=s^{-1}\left( {\displaystyle \sum^{n}_{i=1} }(-1)^{i-1} t^{i}\right)$ and $t^{n+1}=0$, it is clear that every element of $\mathcal{S}_{n}$ is represented as a $\mathbb{Z}$-linear combination of $\{ s^{l}t^{m} \mid l, m\in \mathbb{Z}, \ 0\leq m \leq n\}$. 
We will give a sketch of a proof of the linear independence of $\{ s^{l}t^{m} \mid l, m\in \mathbb{Z}, \ 0\leq m \leq n\}$. 
We follow a ${\rm Poincar\acute{e} }$-Birkhoff-Witt like argument. 
For a monomial $f$ of $T(R)$, regarded as a polynomial ring over three non-commutative variables $s$, $\overline{s}$ and $t$, we assign some sequence of letters (or $0$) denoted by $q(f)$ as follows: 
If $f$ has a part of a form $s\overline{s}$ or $\overline{s}s$, replace this part with $1$, the empty word. 
Repeat this operation whenever there is a part of a form $s\overline{s}$ or $\overline{s}s$. 
The resultant is independent of the order of these replacements. 
Moreover, replace the valuable $\overline{s}$ with $s^{-1}$ as a notation. 
Set $q(f)$ to be this resultant. 
For a monomial $f$ of $T(R)$, this $q(f)$ is represented uniquely by the form of $s^{a_{k}}t^{b_{k}}\cdots s^{a_{1}}t^{b_1}$, where $|a_{k}|, b_{1}\geq 0$ and the other $|a_{i}|, b_i >0$ and the case $k=0$ is regarded as $q(f)=1$. 
We assign an element $L(s^{a_{k}}t^{b_{k}}\cdots s^{a_{1}}t^{b_1})$ of $\mathbb{Z}[X^{\pm1}, Y]/(Y^{n+1})$ to $s^{a_{k}}t^{b_{k}}\cdots s^{a_{1}}t^{b_1}$ by an induction on $\left( k, (|a_{1}|, \dots, |a_{k}|)\right)$ in lexicographical order as follows. 

\begin{itemize}
\item Set $L(1)$ to be $1\in \mathbb{Z}[X^{\pm1}, Y]/(Y^{n+1})$, 
\item set $L(s^{a_{1}}t^{b_{1}})$ to be $X^{a_{1}}Y^{b_1}\in \mathbb{Z}[X^{\pm1}, Y]/(Y^{n+1})$ and
\item for $s^{a_{k}}t^{b_{k}}\cdots s^{a_{1}}t^{b_1}$ with $k\geq 1$, ``choose'' some $l<k$ and set $L(s^{a_{k}}t^{b_{k}}\cdots s^{a_{l+1}}t^{b_{l+1}}s^{a_l}t^{b_{l}}\cdots s^{a_{1}}t^{b_1}) $to be the following. Note that the right hand sides are already defined by the inductive steps.  
\begin{displaymath}
\left \{ 
\begin{array}{l} {\displaystyle \sum^{n}_{i_{1}=1}} \cdots {\displaystyle \sum^{n}_{i_{b_{l+1}=1}}} L(s^{a_{k}}t^{b_{k}}\cdots s^{a_{l+1}+1}t^{(i_{1}+\cdots +i_{b_{l+1}})}s^{a_{l}-1}t^{b_{l}}\cdots s^{a_1}t^{b_1})\ \  \ \ (a_{l}>0) \\ 
(-1)^{b_{l+1}} {\displaystyle \sum^{n}_{i_{1}=1}} \cdots {\displaystyle \sum^{n}_{i_{b_{l+1}=1}}}(-1)^{(i_{1}+\cdots +i_{b_{l+1}})} L(s^{a_{k}}t^{b_{k}}\cdots s^{a_{l+1}-1}t^{(i_{1}+\cdots +i_{b_{l+1}})}s^{a_{l}+1}t^{b_{l}}\cdots s^{a_1}t^{b_1})\ \  \ \ (a_{l}<0) 
\end{array} \right. 
\end{displaymath} 

\end{itemize}

We should show that the above is independent of the choice of $l$, which we omit. 
Through proving this independence, we must prove that $L(s^{a_{k}}t^{b_{k}}\cdots s^{a_{1}}t^{b_1}) =0$ if ${\displaystyle \sum^{k}_{i=1}t^{b_{i}}}>n$ by the same induction. 
Thus we have an element $L\circ q (f) \in \mathbb{Z}[X^{\pm1},Y]/(Y^{n+1})$ for a monomial $f$ of $T(R)$. 
Extending this linearly, we get a $\mathbb{Z}$-linear map $L': T(R)\longrightarrow \mathbb{Z}[X^{\pm1}, Y]/(Y^{n+1})$. 
It should be checked that every element of the two-sided ideal $I_n$ is mapped to $0\in \mathbb{Z}[X^{\pm1},Y]/(Y^{n+1})$ by $L'$, which we omit again. 
Then $L'$ induces a $\mathbb{Z}$-linear map from $\mathcal{S}_n$ to $\mathbb{Z}[X^{\pm1},Y]/(Y^{n+1})$. 
Under this map, the set $\{ s^{l}t^{m} \mid l, m\in \mathbb{Z}, \ 0\leq m \leq n\}$ is mapped to the set $\{ X^{l}Y^{m} \mid l, m\in \mathbb{Z}, \ 0\leq m \leq n\}$, which consists of linearly independent elements of $\mathbb{Z}[X^{\pm1},Y]/(Y^{n+1})$. 
This implies the linear independence of $\{ s^{l}t^{m} \mid l, m\in \mathbb{Z}, \ 0\leq m \leq n\}$. 
\end{proof}

\subsection{Unit elements of $\mathcal{S}_{n}$}
We will characterize unit elements of ${\mathcal{S}}_{n}$. 
Recall that an element $x$ of a not necessarily commutative ring is a unit element if there exists an element $y$, called the inverse of $x$, satisfying $xy=yx=1$. 
Such $y$ is unique for $x$ and is denoted by $x^{-1}$. 

\begin{defini}
Let $n$ and $m$ be non-negative integers with $n\leq m$. 
Since two-sided ideals $I_{n}$ and $I_{m}$ of $T(R)$ satisfies $I_{m}\subset I_{n}$, there exists a ring homomorphism from $\mathcal{S}_{m}$ to $\mathcal{S}_{n}$ induced by the identity map of $T(R)$. 
This map is denoted by $\pi_{n,m}: \mathcal{S}_{m}\longrightarrow \mathcal{S}_{n}$, and called projection. 
Under the bases of Proposition~\ref{basis}, the image of an element of $\mathcal{S}_{m}$ under $\pi_{n,m}$ is obtained by discarding the terms whose powers of $t$ is grater than $n$. 
It can be seen that for non-negative integers $n$, $m$ and $l$ with $n\leq m\leq l$, the equation $\pi_{n,l}=\pi_{n,m}\circ \pi_{m,l}$ holds. 
\end{defini}

\begin{rmk}\label{submodule}
For non-negative integers $n$ and $m$ with $n\leq m$, the basis of Proposition~\ref{basis} for $\mathcal{S}_{n}$ is a subset of that for $\mathcal{S}_{m}$. 
Thus $\mathcal{S}_{n}$ can be regarded as a submodule of $\mathcal{S}_{m}$. 
However, $\mathcal{S}_{n}$ is not a subring of $\mathcal{S}_{m}$. 
In fact, the product of $t$ and $t^{n}$ in $\mathcal{S}_{n}$ is $0$ but that in $\mathcal{S}_{n+1}$ is not $0$. 
\end{rmk}

\begin{lem}\label{unit}
The element $x\in \mathcal{S}_n$ is a unit element if and only if $\pi_{0,n}(x)$ is a unit element of $\mathcal{S}_{0}=\mathbb{Z}[s^{\pm1}]$. 
\end{lem}
\begin{proof}
Only if part is clear by applying $\pi_{0,n}$ to an equation $xy=1$ for $x,y\in \mathcal{S}_{n}$. 
We prove if part. 
We prove by induction. The case $n=0$ is clear. 
Suppose that the statement is true for some $k\geq 0$. 
Let $x\in \mathcal{S}_{k+1}$ be an element such that $\pi_{0,k+1}(x)$ is a unit element of $\mathcal{S}_{0}=\mathbb{Z}[s^{\pm1}]$. 
Set $x'$ to be $\pi_{k,k+1}(x)$. 
Then $x'$ is a unit element of $\mathcal{S}_{n}$ by the induction step since $\pi_{0,k}(x')=\pi_{0,k+1}(x)$ is a unit element. 
Take the inverse $y'\in \mathcal{S}_{k}$ of $x'$. 
As in Remark~\ref{submodule}, these $x', y'\in \mathcal{S}_{k}$ can be regarded as elements of $\mathcal{S}_{k+1}$. 
In $\mathcal{S}_{k+1}$, the element $x$ and the products $x'y'$ and $y'x'$ can be represented as $x'+ft^{k+1}$, $1+gt^{k+1}$ and $1+ht^{k+1}$ for some $f, g, h\in \mathbb{Z}[s^{\pm1}]$ respectively. 
Then $xy=1$ holds in $\mathcal{S}_{k+1}$ by setting $y$ to be $y'-\left(\pi_{0,k+1}(x)\right)^{-1}\left(\pi_{0,k}(y')\cdot f+g\right)t^{k+1}$. This computation is guaranteed by Remark~\ref{degree}. 
It remains to show that $y x=1$ holds. 
By computation, $yx=\left( y'-\left(\pi_{0,k+1}(x)\right)^{-1}\left(\pi_{0,k}(y')\cdot f+g\right)t^{k+1} \right)\cdot (x'+ft^{k+1})=1+(h-g)t^{k+1}$ holds in $\mathcal{S}_{k+1}$. 
Thus we have $x(yx-1)=x(h-g)t^{k+1}=\pi_{0,k+1}(x)(h-g)t^{k+1}$. 
On the other hand, $x(yx-1)=(xy)x-x=0$. 
This implies $h=g$ and $yx=1$ in $\mathcal{S}_{k+1}$. 
\end{proof}

\section{Semiquandle structure on $\mathcal{S}_{n}$}
In \cite{semiquandle}, the algebraic structure {\it semiquandle} is defined. 
It is variant of quandle or biquandle. 
As with quandle and biquandle, this semiquandle structure corresponds to the axiomatization of the transformations preserving the isotopy type of some variant of knots, {\it flat virtual knots}. See Section~\ref{sec_flatvirtuallinks}. 
In this section, we recall the definition of semiquandles and give our $\mathcal{S}_{n}$ the structure of semiquandle. 

\begin{defini}\label{def_semiquandle}
(Definition 1 of \cite{semiquandle})\\
A {\it Semiquandle} is a set $X$ with two binary operations $(x,y)\longmapsto x\triangleleft_{o}y$ and $(x,y)\longmapsto x\triangleleft_{u}y$ such that, for all $x,\ y,\ z\in X$;
\begin{itemize}
\item[(0)] there are unique $v$ and $w\in X$ with $x=v\triangleleft_{o}y$ and $x=w\triangleleft_{u}y$,
\item[(1)] $x\triangleleft_{u}y=y$ if and only if $y\triangleleft_{o}x=x$,
\item[(2)] $(x\triangleleft_{u}y)\triangleleft_{o}(y\triangleleft_{o}x)=x$ and $(x\triangleleft_{o}y)\triangleleft_{u}(y\triangleleft_{u}x)=x$, and
\item[(3)] $\left( x\triangleleft_{o}y \right)\triangleleft_{o}z=\left( x\triangleleft_{o}(z\triangleleft_{u}y) \right)\triangleleft_{o}(y\triangleleft_{o}z)$, \ $\left( y\triangleleft_{u}x \right)\triangleleft_{o} \left(z\triangleleft_{u}(x\triangleleft_{o}y) \right)=\left( y\triangleleft_{o}z \right)\triangleleft_{u} \left(x\triangleleft_{o}(z\triangleleft_{u}y) \right)$, and 
$\left( z\triangleleft_{u}(x\triangleleft_{o}y) \right)\triangleleft_{u}(y\triangleleft_{u}x)=\left( z\triangleleft_{u}y \right)\triangleleft_{u}x$.
\end{itemize}
\end{defini}

\begin{prop}
$\mathcal{S}_{n}$ becomes a semiquandle by defining $x\triangleleft_{o}y=s\cdot x+t\cdot y$ and $x\triangleleft_{u}y=s^{-1}(1-t^{2})\cdot x-s^{-1}ts\cdot y$ for all $x,y\in \mathcal{S}_{n}$.
\end{prop}

\begin{proof}This follows by the following computations. 
Note that $ts=s{\displaystyle \sum^{n}_{i=1} t^{i}}=s{\displaystyle \sum^{n+1}_{i=1} t^{i}}=st{\displaystyle \sum^{n}_{i=0} t^{i}}=st(1-t)^{-1}$ and $ts^{-1}=s^{-1}t(1+t)^{-1}$ hold.
\begin{itemize}
\item[(0)] $(s^{-1}x-s^{-1}ty)\triangleleft_{o}y=x$ and $\left( (1-t^{2})^{-1}sx+(1-t^{2})^{-1}tsy \right)\triangleleft_{u}y=x$\\ 
We can see easily that $\cdot \triangleleft_{o}y$ and $\cdot \triangleleft_{u}y$ are injective. 

\item[(1)]
                \begin{align}
                        &x\triangleleft_{u}y=y \iff s^{-1}(1-t^{2})x=(1+s^{-1}ts)y  \iff (1+t)(1-t)x=(1+t)sy \iff (1-t)x=sy    \notag  \\
                        &y\triangleleft_{o}x=x \iff sy=(1-t)x \notag
                \end{align}

\item[(2)]
                \begin{align}
                        (x\triangleleft_{u}y)\triangleleft_{o}(y\triangleleft_{o}x)&=\left( s^{-1}(1-t^2)x-s^{-1}tsy \right) \triangleleft_{o}(sy+tx) =s\left( s^{-1}(1-t^2)x-s^{-1}tsy\right) +t(sy+tx) =x  \notag  \\
                        (x\triangleleft_{o}y)\triangleleft_{u}(y\triangleleft_{u}x)&= (sx+ty)\triangleleft_{u} \left( s^{-1}(1-t^2)y-s^{-1}tsx \right)  \notag \\
&=s^{-1}(1-t^2)(sx+ty)-s^{-1}ts\left( s^{-1}(1-t^2)y-s^{-1}tsx \right) =x \notag  
                \end{align}

\item[(3)] 
                \begin{align}
                        (x\triangleleft_{o}y)\triangleleft_{o}z&=s(sx+ty)+tz=s^{2}x+sty+tz  \notag  \\
                        \left( x\triangleleft_{o}(z\triangleleft_{u}y) \right)\triangleleft_{o}(y\triangleleft_{o}z)&= s\left( sx+t\left( s^{-1}(1-t^{2})z-s^{-1}tsy \right)  \right) +t(sy+tz) \notag \\
&=s^{2}x+\left( ts-sts^{-1}ts \right)y +\left( sts^{-1}(1-t^{2})+t^2 \right)z \notag  \\
&=s^{2}x+\left( st(1-t)^{-1}-st^{2}(1-t)^{-1} \right)y+\left( t(1+t)^{-1}(1-t^2)+t^{2} \right)z \notag \\
&=s^{2}x+sty+tz \notag 
                \end{align}

                \begin{align}
                        \left( y\triangleleft_{u}x \right)\triangleleft_{o} \left(z\triangleleft_{u}(x\triangleleft_{o}y) \right)&=s\left( s^{-1}(1-t^2)y-s^{-1}tsx \right) +t\left( s^{-1}(1-t^{2})z-s^{-1}ts\left( sx+ty \right)  \right)  \notag  \\
&=-(ts+ts^{-1}ts^2)x+\left( (1-t^2)-ts^{-1}tst \right)y+ts^{-1}(1-t^2)z \notag \\
&=- t\left( 1+t(1-t)^{-1}  \right)sx+\left( (1-t^2)-t^{2}(1-t)^{-1}t \right)y+s^{-1}t(1+t)^{-1}(1-t^2)z \notag \\
&=- t\left( 1+t{\displaystyle \sum^{n}_{i=0}t^{i}}  \right)sx+\left( (1-t^2)-t^{3}{\displaystyle \sum^{n}_{i=0}t^{i}} \right)y+s^{-1}t(1-t)z \notag \\
&=- \left({\displaystyle \sum^{n}_{i=1}t^{i}}\right)sx+\left( 1-{\displaystyle \sum^{n}_{i=2}t^{i}} \right)y+s^{-1}t(1-t)z \notag \\
                        \left( y\triangleleft_{o}z \right)\triangleleft_{u} \left(x\triangleleft_{o}(z\triangleleft_{u}y) \right) &=s^{-1}(1-t^{2})(sy+tz)-s^{-1}ts\left( sx+t\left(  s^{-1}(1-t^2)z-s^{-1}tsy \right)  \right) \notag \\
&=-s^{-1}ts^{2}x+\left( s^{-1}(1-t^2)s+s^{-1}tsts^{-1}ts  \right)y+ \left( s^{-1}(1-t^2)t-s^{-1}tsts^{-1}(1-t^2)   \right)z \notag \\
&=-s^{-1}ts^{2}x+s^{-1}\left( 1-t^2+tsts^{-1}t  \right)sy+ s^{-1}t\left( (1-t^2)-t(1+t)^{-1}(1-t^2)   \right)z \notag \\
&=-\left( {\displaystyle \sum^{n}_{i=1}t^{i}} \right)sx+s^{-1}\left( 1-t^2+t\left( {\displaystyle \sum^{n}_{i=1}(-1)^{i-1}t^i} \right)t  \right)sy+ s^{-1}t\left( (1-t^2)-t(1-t)   \right)z \notag \\
&=-\left( {\displaystyle \sum^{n}_{i=1}t^{i}} \right)sx+s^{-1}\left( 1-t+{\displaystyle \sum^{n}_{i=1}(-1)^{i-1}t^i}  \right)sy+ s^{-1}t\left( 1-t   \right)z \notag \\
&=-\left( {\displaystyle \sum^{n}_{i=1}t^{i}} \right)sx+s^{-1}\left( 1-t+sts^{-1} \right)sy+ s^{-1}t\left( 1-t   \right)z \notag \\
&=-\left( {\displaystyle \sum^{n}_{i=1}t^{i}} \right)sx+\left( 1-s^{-1}ts+t \right)y+ s^{-1}t\left( 1-t   \right)z \notag \\
&=- \left({\displaystyle \sum^{n}_{i=1}t^{i}}\right)sx+\left( 1-{\displaystyle \sum^{n}_{i=2}t^{i}} \right)y+s^{-1}t(1-t)z \notag
                \end{align}

                \begin{align}
                  \left( z\triangleleft_{u}(x\triangleleft_{o}y) \right)\triangleleft_{u}(y\triangleleft_{u}x)&=s^{-1}(1-t^{2})\left(  s^{-1}(1-t^2)z-s^{-1}ts( sx+ty ) \right)-s^{-1}ts \left( s^{-1}(1-t^2)y-s^{-1}tsx \right) \notag \\
&=-\left( s^{-1}(1-t^2)s^{-1}ts^2  -s^{-1}t^{2}s \right)x-\left( s^{-1}(1-t^2)s^{-1}tst+s^{-1}t(1-t^2)   \right)y \notag \\
& \hspace{7cm} +s^{-1}(1-t^2)s^{-1}(1-t^2)  z \notag \\
&=-s^{-1}\left( (1-t^2)s^{-1}ts  -t^{2} \right)sx-s^{-1}(1-t^2)\left( s^{-1}ts+1 \right)ty \notag \\
& \hspace{7cm} +s^{-1}(1-t^2)s^{-1}(1-t^2) z \notag \\
&=-s^{-1}\left( (1-t^2)t(1-t)^{-1}  -t^{2} \right)sx-s^{-1}(1-t^2)\left( \left({\displaystyle \sum^{n}_{i=1}t^{i}}\right)+1 \right)ty \notag \\
& \hspace{7cm} +s^{-1}(1-t^2)s^{-1}(1-t^2) z \notag  \\
&=-s^{-1}tsx-s^{-1}(1+t)ty + s^{-1}(1-t^2)s^{-1}(1-t^2) z \notag \\
                   \left( z\triangleleft_{u}y \right)\triangleleft_{u}x&=s^{-1}(1-t^2)\left( s^{-1}(1-t^2)z-s^{-1}tsy   \right)-s^{-1}tsx \notag \\
&=-s^{-1}tsx-s^{-1}(1-t^2)s^{-1}tsy+s^{-1}(1-t^2)s^{-1}(1-t^2) z \notag \\
&=-s^{-1}tsx-s^{-1}(1-t^2)t(1-t)^{-1}y+s^{-1}(1-t^2)s^{-1}(1-t^2) z \notag \\
&=-s^{-1}tsx-s^{-1}(1+t)ty + s^{-1}(1-t^2)s^{-1}(1-t^2) z \notag 
                \end{align}

\end{itemize}

\end{proof}

\begin{rmk}
For non-negative integers $n$ and $m$ with $n\leq m$, the projection $\pi_{n,m}$ preserves the semiquandle structures of $\mathcal{S}_{m}$ and $\mathcal{S}_{n}$. 
That is, $\pi_{n,m}\left(x\triangleleft_{o}y\right)=\pi_{n,m}(x)\triangleleft_{o}\pi_{n,m}(y)$ and $\pi_{n,m}\left(x\triangleleft_{u}y\right)=\pi_{n,m}(x)\triangleleft_{u}\pi_{n,m}(y)$ hold for all $x,y \in \mathcal{S}_{m}$. 
\end{rmk}

\section{Flat virtual links}\label{sec_flatvirtuallinks}
In this section, we recall the definition of flat virtual links and give some preliminaries we need. In this paper, flat virtual links are always oriented. 

\subsection{Definition of flat virtual links}
\begin{defini}\label{def_flatvirtuallinkdiagram}
A {\it flat virtual link diagram} is an immersion (i.e. locally injection) of disjoint union of oriented circles $S^{1}\amalg \cdots \amalg S^{1}$ into $\mathbb{R}^{2}$ such that the number of the preimage of every point of the range $\mathbb{R}^{2}$ is less than three and that the images of this immersion are transversal at the double points, with each double point being given the label ``flat'' or ``virtual'' as in Figure~\ref{flat_virtual}. 
We consider flat link diagrams as the images in $\mathbb{R}^{2}$. 
If the domain of a flat virtual link diagram is a single $S^{1}$, this is called a {\it flat virtual knot diagram}. 
Preimages of all flat crossings in a flat virtual link diagram cut the domain $S^{1}\amalg \cdots \amalg S^{1}$. 
We call the image of each connected component of this cut domain an arc of the flat virtual link diagram. 
\end{defini}

\begin{figure}[htbp]
 \begin{center}
  \includegraphics[width=50mm]{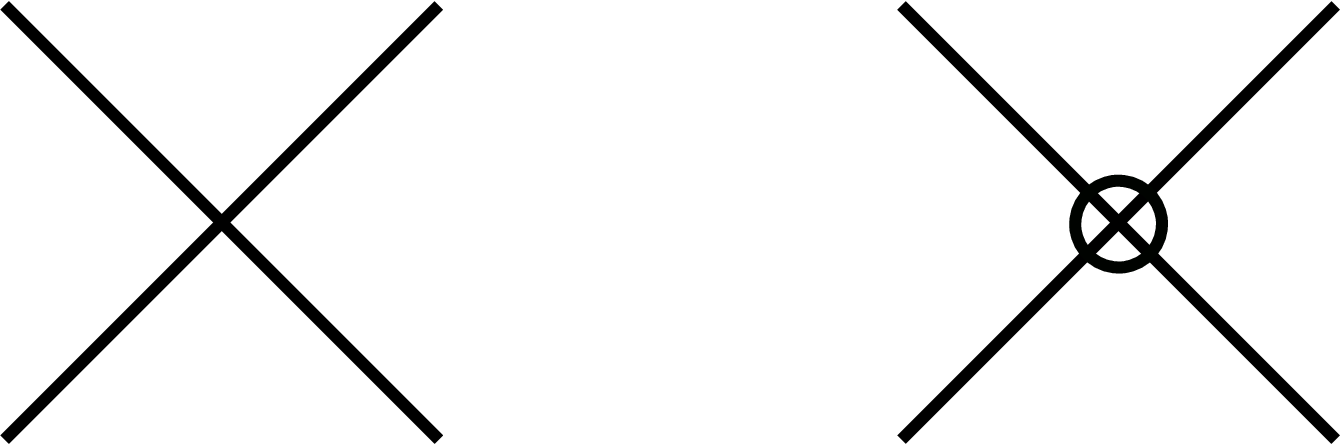}
 \end{center}
 \caption{Left: flat crossing \ \ \ \ \ Right: virtual crossing}
 \label{flat_virtual}
\end{figure}

\begin{figure}[htbp]
 \begin{center}
  \includegraphics[width=60mm]{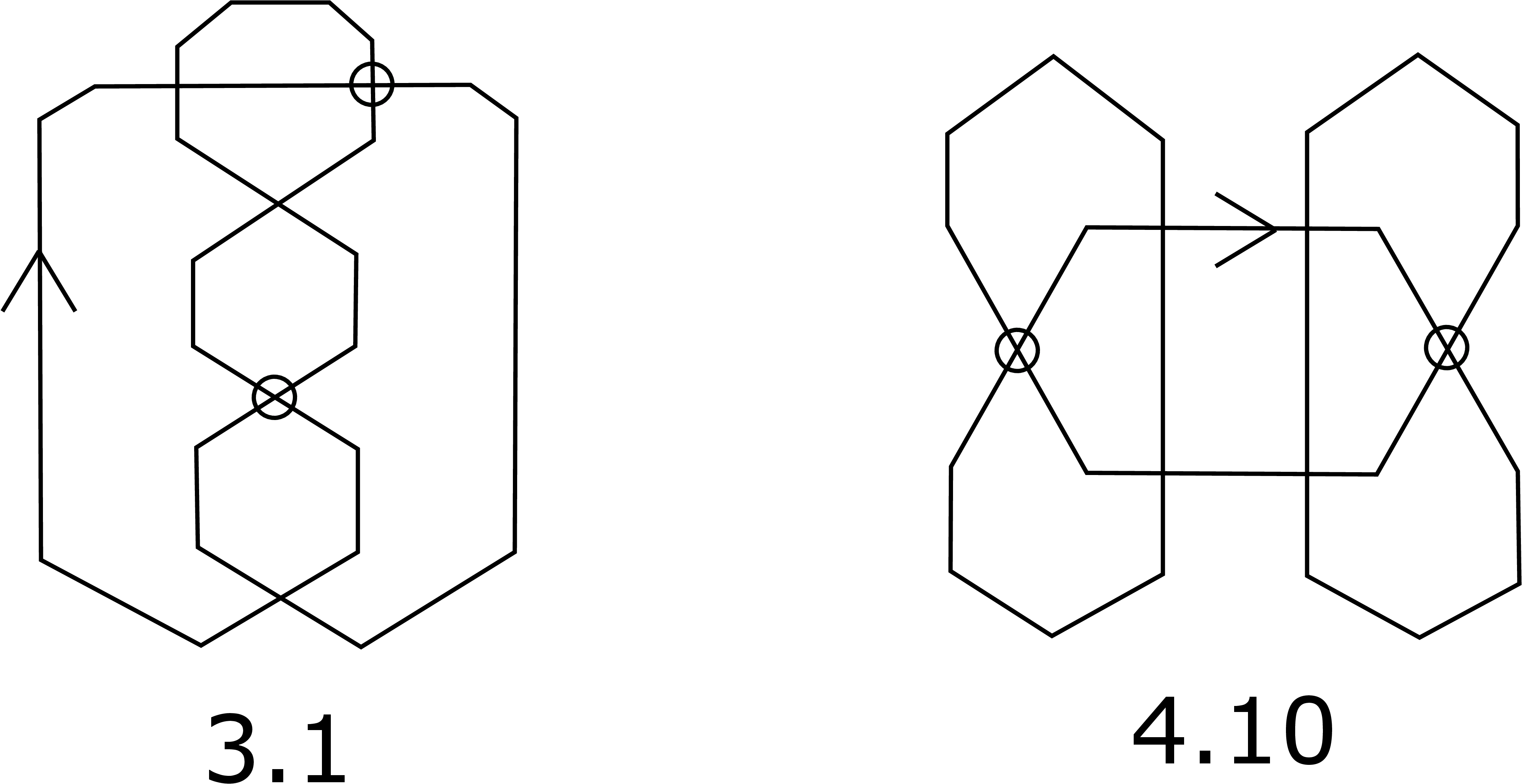}
 \end{center}
 \caption{Examples of flat virtual knot diagrams.}
 \label{example_flatvirtual}
\end{figure}

\begin{defini}
Two flat virtual link diagrams are said to be homotopic if they are related by a finite sequence of {\it flat Reidemeister moves} shown in Figure~\ref{flat_reidemeister} along with planar isotopies. In Figure~\ref{flat_reidemeister}, orientations are arbitrary. 
Each equivalence class of the quotient of the set of flat virtual link diagrams by the equivalence relation of homotopics is called a {\it flat virtual link}. 
The equivalence classes of flat virtual knot diagrams are called {\it flat virtual knots}. 
\end{defini}

\begin{figure}[htbp]
 \begin{center}
  \includegraphics[width=80mm]{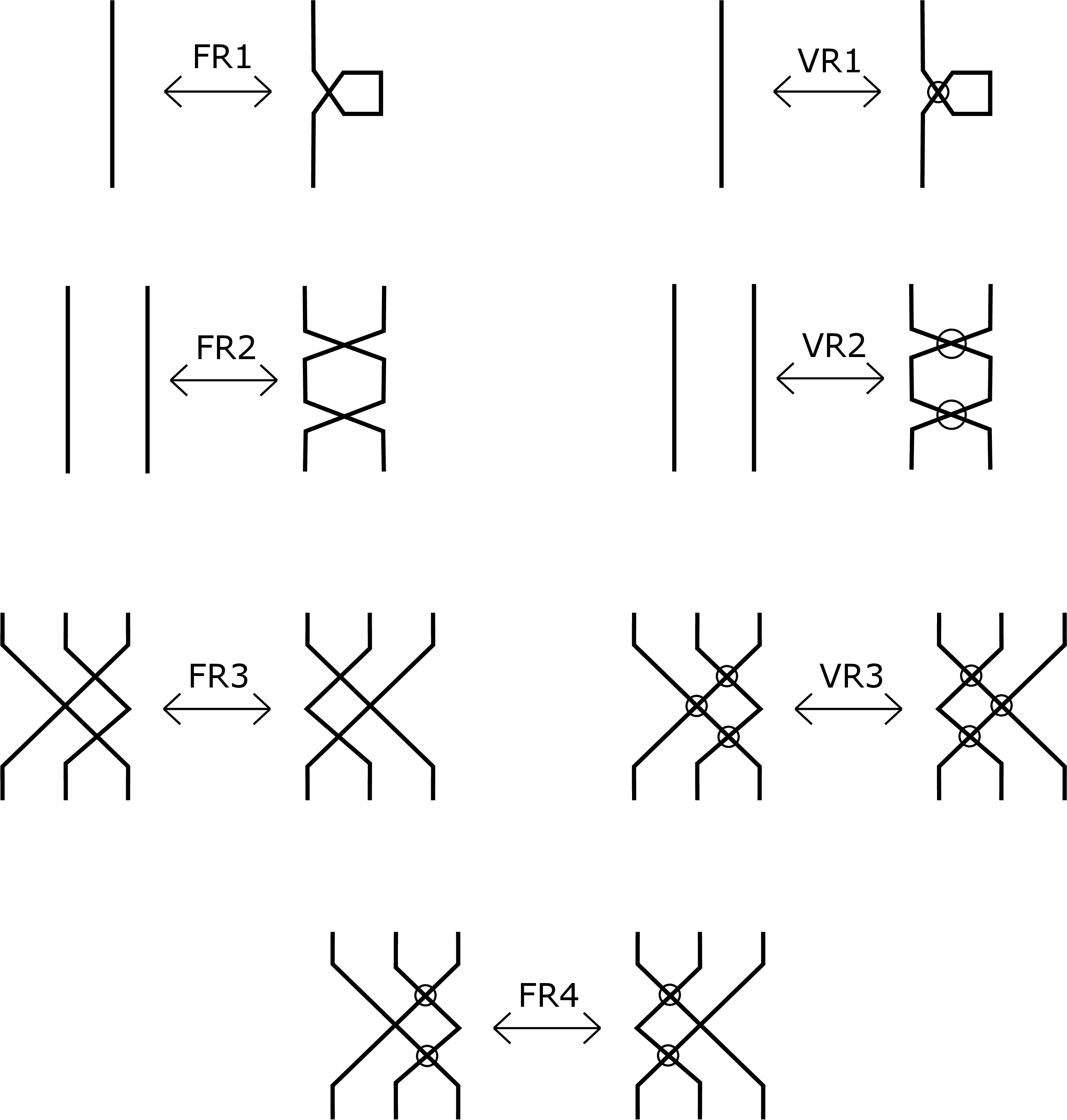}
 \end{center}
 \caption{Flat Reidemeister moves, where orientations are arbitrary.}
 \label{flat_reidemeister}
\end{figure}

By Theorem 1.2 of \cite{minimalgenerating}, all ordinary oriented Reidemeister moves for link diagrams are generated by five moves. 
By applying these arguments to moves FR1, FR2 and FR3, we have the following. 

\begin{fact}
All of moves FR1, FR2 and FR3 are generated by the five moves in Figure~\ref{fivemoves}.
\end{fact}

\begin{figure}[htbp]
 \begin{center}
  \includegraphics[width=60mm]{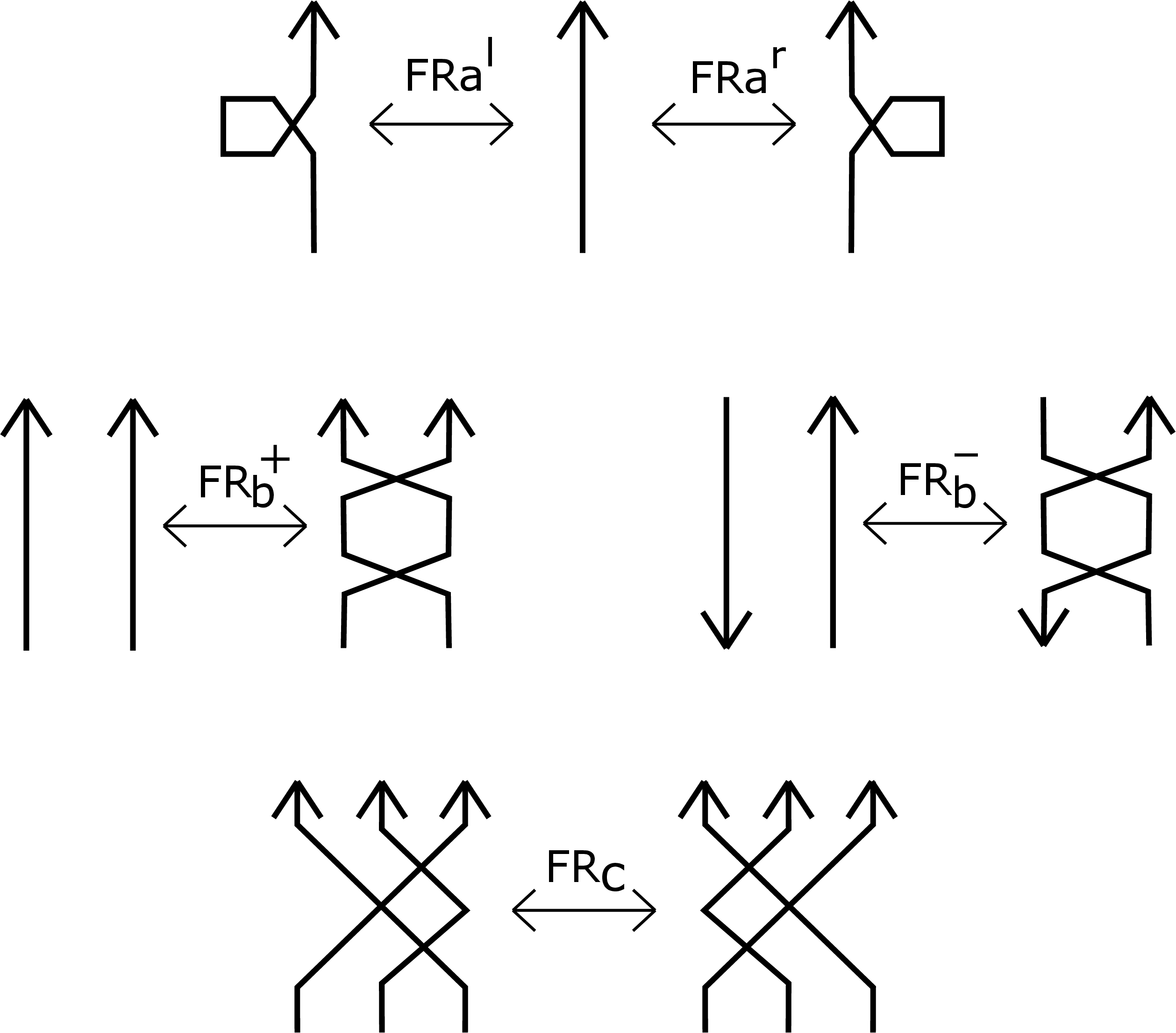}
 \end{center}
 \caption{Five moves which generate moves FR1, FR2 and FR3.}
 \label{fivemoves}
\end{figure}

\begin{rmk} \label{invariance_coloring}
The properties of semiquandles in Definition~\ref{def_semiquandle} translate the five moves ${{\rm FRa}}^{r}$, ${{\rm FRa}}^{l}$, ${{\rm FRb}}^{+}$, ${{\rm FRb}}^{-}$ and FRc in Figure~\ref{fivemoves} into algebraic axioms. 
Fix a semiquandle $X$. 
For a flat virtual link diagram $D$, its $X$-coloring is a map from the set of arcs (Definition~\ref{def_flatvirtuallinkdiagram}) to $X$ satisfying the relation described in Figure~\ref{coloring_crossing} at each flat crossing. 
By definition, virtual crossings are not concerned for $X$-colorings. 
Thus for checking the invariance of cardinality of the set of $X$-colorings for a flat virtual link, it is enough to check the invariance of this for the five moves. 
This is done as in Figure~\ref{coloring_invariance}. 
For example, in the top right of Figure~\ref{coloring_invariance}, $y=x\triangleleft_{u}y$ must hold. 
The property (1) in Definition~\ref{def_semiquandle} forces the coloring of the other endpoint to be $x$. 
Moreover, this $y$ is unique for given $x$ by the property (0) in Definition~\ref{def_semiquandle}. 
Thus the set of $X$-colorings of two flat virtual link diagrams related by ${{\rm FRa}}^{r}$ are one-to-one. 
The conditions for the endpoints matching under ${{\rm FRb}}^{+}$, ${{\rm FRb}}^{-}$ and FRc moves are equivalent to the properties (2) and (3) in Definition~\ref{def_semiquandle}. 
\end{rmk}

\begin{figure}[htbp]
 \begin{center}
  \includegraphics[width=30mm]{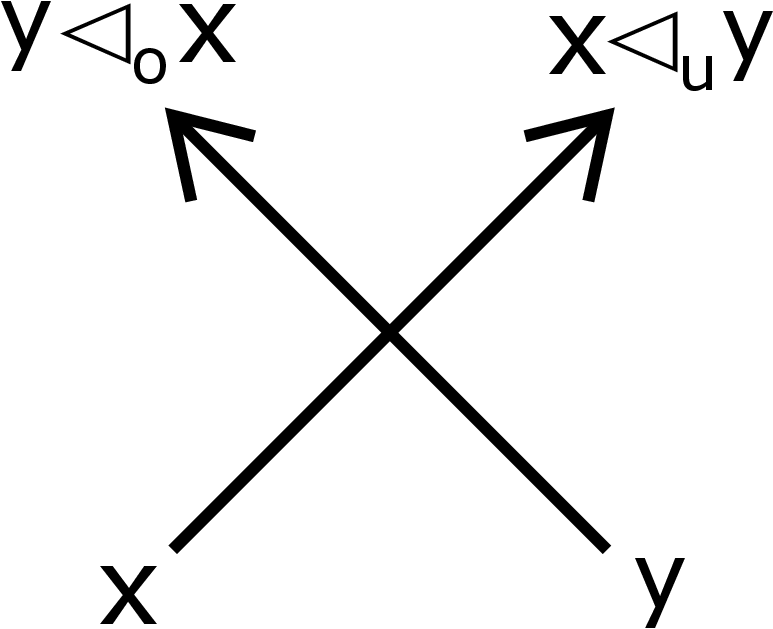}
 \end{center}
 \caption{Coloring relation}
 \label{coloring_crossing}
\end{figure}

\begin{figure}[htbp]
 \begin{center}
  \includegraphics[width=60mm]{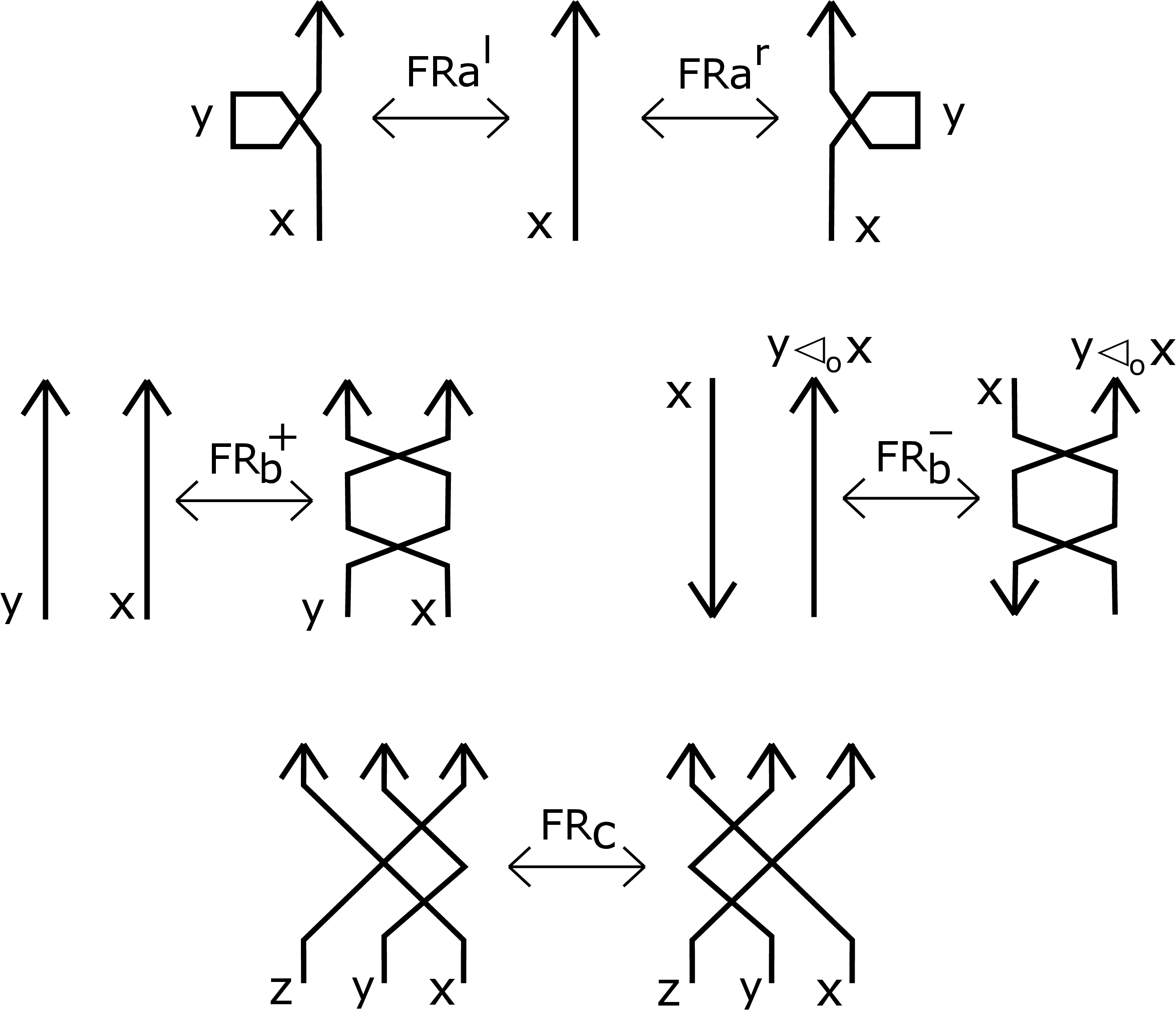}
 \end{center}
 \caption{Coloring on five moves}
 \label{coloring_invariance}
\end{figure}

\subsection{Gauss diagrams and u-polynomials}
In this subsection, we consider flat virtual knots. 
There is some useful presentation for flat virtual knots, called the {\it Gauss diagram}. 
Using this, we can compute some invariant of flat virtual knots, called the {\it u-polynomial}.  

\subsubsection{Gauss diagrams}
We start with the definition of Gauss diagrams and see their correspondences with flat virtual knots. 
\begin{defini}
A Gauss diagram is a counter-clockwise oriented circle on a plane with finitely many dashed oriented arrows whose endpoints are on the circle and pairwise distinct. 
\end{defini}

For a flat virtual knot diagram, we can construct a (unique) Gauss diagram as follows: 
Identify the domain $S^{1}$ of the flat virtual knot diagram with the counter-clockwise oriented circle on a plane. 
For each flat crossing of the flat virtual knot diagram, there are two points $p_1$ and $p_2$ on the circle which are identified by the immersion, which is in the definition of the flat virtual knot diagram. 
Let $v_1$ and $v_2$ be the image of velocity vectors of circle under the immersion at $p_1$ and $p_2$, respectively. 
Write a dashed arrow oriented from $p_{1}$ to $p_2$ if $<v_1,v_2>$ is positive orientation of the target space $\mathbb{R}^{2}$ of the immersion, and write it oriented from $p_{2}$ to $p_1$ otherwise. 

Conversely, we can construct a flat virtual knot diagram from a given Gauss diagram as follows: 
For each dashed oriented arrow in the Gauss diagram, put oriented flat crossing on $\mathbb{R}^{2}$. 
Then connect them according to the order given by the Gauss diagram.
We should connect them so that the orientation rule in the previous paragraph are satisfied. 
Along this connection, new crossings may occur. 
All such crossings are decorated with virtual crossings. 
By this way, we get (the image of) a flat virtual knot diagram from a given Gauss diagram. 
In this construction, there are many ways of connecting flat crossings. 
But two flat virtual knot diagrams for a given Gauss diagram with different connections are related by moves VR1, VR2, VR3 and FR4 in Figure~\ref{flat_reidemeister}. 
For other moves FR1, FR2 and FR3, there are counterpart moves in Gauss diagrams. 
See Figure 5 and 6 in \cite{chen2} for example. 
By the observation above, we see that Gauss diagrams modulo the Gauss diagram version of moves FR1, FR2 and FR3 are one-to-one to flat virtual knots.  

\begin{figure}[htbp]
 \begin{center}
  \includegraphics[width=50mm]{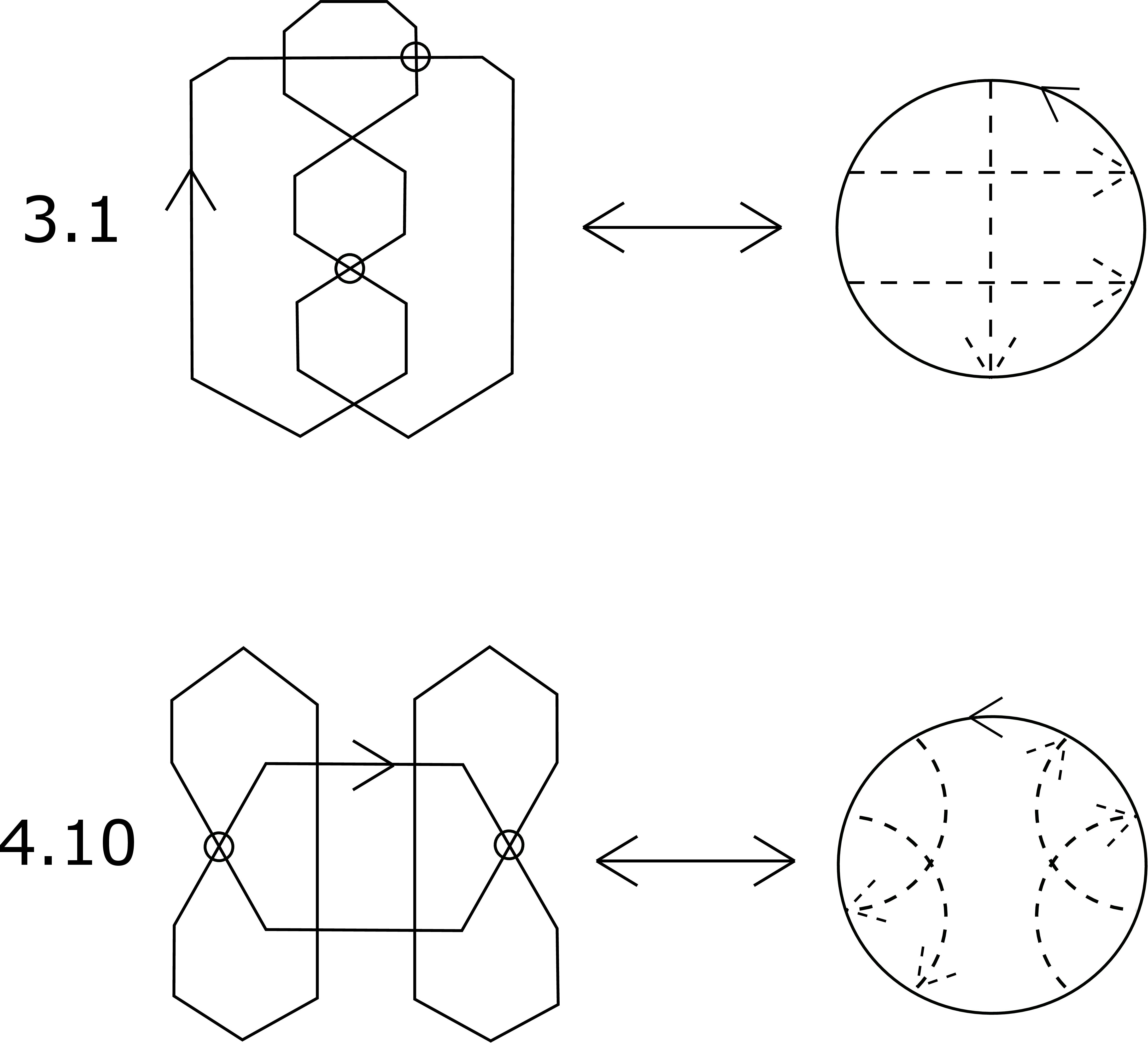}
 \end{center}
 \caption{Examples of Gauss diagrams of flat virtual knots.}
 \label{example_gauss}
\end{figure}

\subsubsection{u-polynomials} \label{u-polynomial}
{\it u-polynomials} are invariants of flat virtual knots, which are introduced in \cite{upolynomial}. 
They can be computed through Gauss diagrams. 

\begin{defini}
For a Gauss diagram $D$, assign each tail (i.e. starting point) of each dashed oriented arrow the sign $+$ and each head (i.e. terminal point) of it the sign $-$. 
The set of dashed oriented arrows of $D$ is denoted by $arr(D)$. 
For $e\in arr(D)$, set $n(e)$ to be the number of the sum of the signs on the part of the circle of $D$ which is the right side of $e$, where we see the arrow $e$ so that the head is above and the tail is below. 
\end{defini} 

\begin{defini}
Let $D$ be a Gauss diagram. 
Define $u_{D}(s)$ as $\displaystyle \sum_{e\in arr(D)} sign\left( n(e) \right)s^{|n(e)|}\in \mathbb{Z}[s]$, where $sign(0)$ is regarded as $0$. 
\end{defini}

\begin{fact}
Let $D$ and $D'$ be two Gauss diagrams related by the Gauss diagram version of moves FR1, FR2 and FR3. 
Then $u_{D}(s)=u_{D'}(s)$ holds. 
\end{fact}

This can be checked by the direct computation. 
By this fact, $u_{D}(s)$ is an invariant of the flat virtual knot $\alpha$ one of whose representative is $D$. 
Denote it by $u_{\alpha}(s)$ and call it {\it u-polynomial} of $\alpha$.

\section{Quasideterminant}
In this section, we review the definition of {\it quasideterminant} and its property we will use. 
Quasideterminant is defined in \cite{quasideterminant}. 
It is defined on square matrix of not necessarily commutative coefficients, and is a counterpart of the ordinary determinant. 
Unlike the ordinary determinant, not only a matrix but also one entry are needed to define the quasideterminant. 
Moreover, quasideterminant is not always defined. 

\subsection*{Notation}
In this section, we work on not necessarily commutative ring $\mathcal{R}$ with unit $1$, fix positive integer $m$, and consider $m$-square matrices with $\mathcal{R}$-entries. 
\begin{itemize}
\item $I$ denotes the identity matrix, and $E_{i,j}$ denotes the matrix whose $(i,j)$-entry is $1$ and the others are $0$ for $1\leq i,j \leq m$. 
\item Set $M_{i,j}(\mu)$ to be $I+\mu E_{i,j}$ for $\mu \in \mathcal{R}$ and for $1\leq i\neq j\leq m$. 
\item Set $X_{i}(u)$ to be $I-E_{i,i}+uE_{i,i}$ for a unit element $u\in \mathcal{R}$ and for $1\leq i \leq m$. 
\item Set $P_{i,j}$ to be $I-E_{i,i}-E_{j,j}+E_{i,j}+E_{j,i}$ for $1\leq i\neq j\leq m$. 
\item For a $m$-square matrix $A$ and integers $1\leq i,j \leq m$, let $\left(A\right)_{i,j}$ denote the $(i,j)$-entry of $A$.
\item For a $m$-square matrix $A$ and integers $1\leq i,j \leq m$, let $A^{i,j}$ denote the $(m-1)$-square matrix obtained from $A$ by deleting the $i$-th row and the $j$-th column.
\item For a $m$-square matrix $A$ and integers $1\leq i,j \leq m$, let ${r_{i}}^{j}(A)$ denote the row vector obtained from the $i$-th row of $A$ by deleting the $j$-th entry.
\item  For a $m$-square matrix $A$ and integers $1\leq i,j \leq m$, let ${c^{i}}_{j}(A)$ denote the column vector obtained from the $j$-th column of $A$ by deleting the $i$-th entry. 
\end{itemize}

\begin{defini}
Let $A$ be a $m$-square matrix. Fix two integers $1\leq i,j \leq m$. 
If $A^{i,j}$ is invertible, define the $(i,j)$-quasideterminant $|A|_{i,j}$ as $(A)_{i,j}-{r_{i}}^{j}(A)\cdot \left(A^{i,j}\right)^{-1} \cdot {c^{i}}_{j}(A)$ when $m>1$ and as $(A)_{1,1}$ when $m=1$. 
\end{defini}

\begin{lem} \label{rowtransformation}
Let $A$ and $B$ be $m$-square matrices such that $B=M_{k,l}(\mu)\cdot A$ for $1\leq k\neq l\leq m$ and $\mu \in \mathcal{R}$. 
Then $|B|_{i,j}=|A|_{i,j}$ for $i\neq l$ and $1\leq j\leq m$ when one of $|A|_{i,j}$ and $|B|_{i,j}$ can be defined.  
\end{lem}

\begin{proof}
Note that $B$ is the result of replacing the $k$-th row of $A$ with the row vector obtained by summing the $k$-th row of $A$ and the result of multiplying $\mu$ to the $l$-th row of $A$ from left. 
Take a pair of integers $i$ and $j$ with $i\neq l$. \\
(1) The case where $i\neq k$ \\
Set $k'$ to be $k$ if $k<i$ and $k-1$ if $k>i$, and set $l'$ to be $l$ if $l<i$ and $k-1$ if $l>i$. 
Then we have $(B)_{i,j}=(A)_{i,j}$, $B^{i,j}=R_{k',l'}(\mu)\cdot A^{i,j}$, ${r_{i}}^{j}(B)={r_{i}}^{j}(A)$ and ${c^{i}}_{j}(B)=R_{k',l'}(\mu)\cdot{c^{i}}_{j}(A)$. 
By this, we have $|B|_{i,j}=|A|_{i,j}$. \\
(2) The case where $i=k$ \\
Set $l'$ to be $l$ if $l<k$ and $l-1$ if $l>k$. 
Then we have $(B)_{k,j}=(A)_{k,j}+\mu(A)_{l,j}$, $B^{k,j}= A^{k,j}$, ${r_{k}}^{j}(B)={r_{k}}^{j}(A)+\mu {r_{l}}^{j}(A)$ and ${c^{k}}_{j}(B)={c^{k}}_{j}(A)$. 
Then $|B|_{k,j}=|A|_{k,j}+\mu (A)_{l,j}-\mu {r_{l}}^{j}(A)\cdot (A^{k,j})^{-1}\cdot {c^{k}}_{j}(A)=|A|_{k,j}$. 
Note that ${r_{l}}^{j}(A)\cdot (A^{k,j})^{-1}$ is the row vector of $m-1$ entries with $l'$-th entry is $1$ and the others are $0$ since ${r_{l}}^{j}(A)$ is the $l'$-th row of $A^{k,j}$ and that the $l'$-th entry of ${c^{k}}_{j}(A)$ is $(A)_{l,j}$. 
\end{proof}

\begin{lem} \label{divide}
 Let $A$ and $B$ be $m$-square matrices such that $B=X_{k}(u)\cdot A$ for a unit element $u\in \mathcal{R}$ and for $1\leq k\leq m$. 
Then $|B|_{i,j}=|A|_{i,j}$ for $i\neq k$ and $1\leq j\leq m$ when one of $|A|_{i,j}$ and $|B|_{i,j}$ can be defined.  
\end{lem}
\begin{proof}
Note that $B$ is the result of multiplying the $k$-th row of $A$ by $u$ from the left. 
Set $k'$ to be $k$ if $k<i$ and $k-1$ if $k>i$. 
Then we have $(B)_{i,j}=(A)_{i,j}$, $B^{i,j}=X_{k'}(u)\cdot A^{i,j}$, ${r_{i}}^{j}(B)={r_{i}}^{j}(A)$ and ${c^{i}}_{j}(B)=X_{k'}(u)\cdot{c^{i}}_{j}(A)$. 
By this, we have $|B|_{i,j}=|A|_{i,j}$. 
\end{proof}

\begin{lem} \label{permutation}

\begin{itemize}
\item[(1)] Let $A$ and $B$ be $m$-square matrices such that $B=P_{k,k+1}\cdot A$ for $1\leq k\leq m-1$. Set $i'$ to be $i$ if $i\neq k, k+1$, to be $k+1$ if $i=k$ and to be $k$ if $i=k+1$. Then $|B|_{i',j}=|A|_{i,j}$ for $1\leq j\leq m$ when one of $|A|_{i,j}$ and $|B|_{i',j}$ can be defined. 
\item[(2)] Let $A$ and $B$ be $m$-square matrices such that $B=A\cdot P_{k,k+1} $ for $1\leq k\leq m-1$. Set $j'$ to be $j$ if $j\neq k, k+1$, to be $k+1$ if $j=k$ and to be $k$ if $j=k+1$. Then $|B|_{i,j'}=|A|_{i,j}$ for $1\leq j\leq m$ when one of $|A|_{i,j}$ and $|B|_{i',j}$ can be defined. 
\end{itemize}
\end{lem}

\begin{proof}
We prove only (1). 
In this case, note that $B$ is obtained from $A$ by exchanging the $k$-th row and the $(k+1)$-th row. \\
(i) The case where $i\neq k, k+1$\\
Set $k'$ to be $k$ if $k<i$ and $k-1$ if $k>i+1$. 
Then we have $(B)_{i',j}=(B)_{i,j}=(A)_{i,j}$, $B^{i',j}=B^{i',j}=P_{k',k'+1}\cdot A^{i,j}$, ${r_{i'}}^{j}(B)={r_{i}}^{j}(B)={r_{i}}^{j}(A)$ and ${c^{i'}}_{j}(B)={c^{i}}_{j}(B)=P_{k',k'+1}\cdot {c^{i}}_{j}(A)$. 
By computation, we get $|B|_{i',j}=|A|_{i,j}$. \\
(ii) The case where $i=k$\\
We have $(B)_{i',j}=(B)_{k+1,j}=(A)_{k,j}$, $B^{i',j}=B^{k+1,j}=A^{k,j}$, ${r_{i'}}^{j}(B)={r_{k+1}}^{j}(B)={r_{k}}^{j}(A)$ and ${c^{i'}}_{j}(B)={c^{k+1}}_{j}(B)={c^{k}}_{j}(A)$. 
By computation, we get $|B|_{i',j}=|A|_{i,j}$. \\
(iii) The case where $i=k+1$\\
We have $(B)_{i',j}=(B)_{k,j}=(A)_{k+1,j}$, $B^{i',j}=B^{k,j}=A^{k+1,j}$, ${r_{i'}}^{j}(B)={r_{k}}^{j}(B)={r_{k+1}}^{j}(A)$ and ${c^{i'}}_{j}(B)={c^{k}}_{j}(B)={c^{k+1}}_{j}(A)$. 
By computation, we get $|B|_{i',j}=|A|_{i,j}$.
\end{proof}

\section{$\mathcal{S}_{n}$-coloring invariant of flat virtual knots}
In this section, we introduce an invariant of flat virtual knot invariant by considering $\mathcal{S}_{n}$-coloring. 

\subsection{Construction}
Let $\alpha$ be a flat virtual knot. Take one of its representative flat virtual knot diagram $D$. 
Let $V_{flat}$ be a set of double points on the image of the immersion which are decorated with flat crossings, and $V_{fake}$ a (possibly empty) set of finite points on the image of the immersion such that $V_{flat}\cap V_{fake}= \emptyset$. 
An element in $V_{flat}$ is called a flat vertex and that in $V_{fake}$ is called a fake vertex. 
Set $V$ to be $V_{flat}\cup V_{fake}$. We suppose that $V$ is not empty by introducing $V_{fake}$ if necessary.  
The preimage of $V$ by the immersion cuts the domain $S^{1}$. 
Abusing terminologies, we call the image of each component of this cut domain also an arc of $(D,V)$. 
Set the number $k$ to be the order of $V$. Note that the number of arcs is also $k$. 

For the pair $(D,V)$, we make a $k$-square matrix with $\mathcal{S}_{n}$ entries as follows: 
``Choose'' one arc of $(D,V)$ and give the arc the label $1$. From this arc, give the other arcs the labels $2,\dots ,k$ in order with respect to the orientation of $D$. 
We call the arc labeled with $i$ {\it the $i$-th arc}. 
At each tail endpoint of $i$-th arc, we take a row vector of $k$-entries, called the $i$-th row vector according to the types of vertex at the tail endpoint as follows, where $(-1)$-th arc or $(-1)$-th entry is regarded as $k$-th arc or $k$-th entry:

\begin{itemize}
\item When the vertex is as in the left of Figure~\ref{ithrow} and $j\neq i-1, i$, the $i$-th row vector is the row vector whose $i$-th entry is $1$, $(i-1)$-th entry is $-s$, $j$-th entry is $-t$ and the other entries are $0$. 
\item When the vertex is as in the left of Figure~\ref{ithrow} and $j= i-1$, the $i$-th row vector is the row vector whose  $i$-th entry is $1$, $(i-1)$-th entry is $-s-t$ and the other entries are $0$. 
\item When the vertex is as in the left of Figure~\ref{ithrow} and $j= i$, the $i$-th row vector is the row vector whose  $i$-th entry is $1-t$, $(i-1)$-th entry is $-s$ and the other entries are $0$. 
\item When the vertex is as in the middle of Figure~\ref{ithrow} and $j\neq i-1, i$, the $i$-th row vector is the row vector whose $i$-th entry is $1$, $(i-1)$-th entry is $-s^{-1}(1-t^2)$, $j$-th entry is $s^{-1}ts$ and the other entries are $0$. 
\item When the vertex is as in the middle of Figure~\ref{ithrow} and $j=i-1$, the $i$-th row vector is the row vector whose $i$-th entry is $1$, $(i-1)$-th entry is $-s^{-1}(1-t^2)+s^{-1}ts$ and the other entries are $0$. 
\item When the vertex is as in the middle of Figure~\ref{ithrow} and $j=i$, the $i$-th row vector is the row vector whose $i$-th entry is $1+s^{-1}ts$, $(i-1)$-th entry is $-s^{-1}(1-t^2)$ and the other entries are $0$. 
\item When the vertex is as in the right of Figure~\ref{ithrow} and $i\neq i-1$, the $i$-th row vector is the row vector whose $i$-th entry is $1$, $(i-1)$-th entry is $-1$ and the other entries are $0$. 
\item When the vertex is as in the right of Figure~\ref{ithrow} and $i= i-1$, the $i$-th row vector is the row vector all of whose entries are $0$. 
\end{itemize}

\begin{figure}[htbp]
 \begin{center}
  \includegraphics[width=80mm]{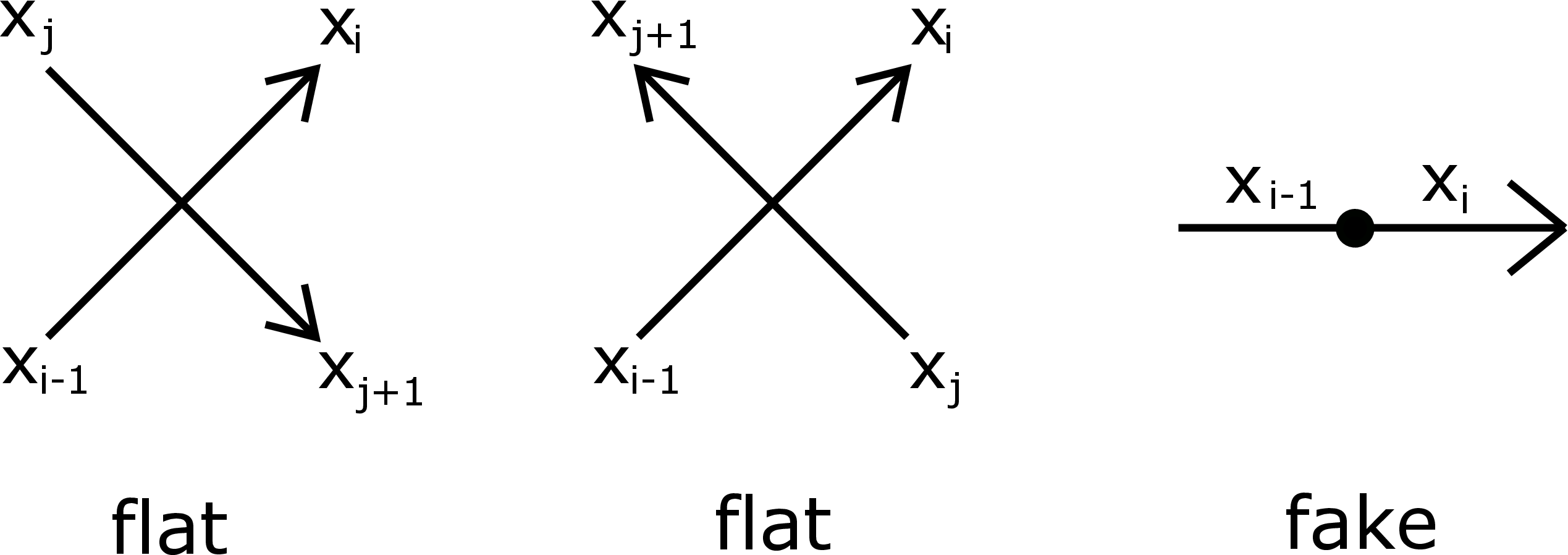}
 \end{center}
 \caption{Cases of vertex at the tail endpoint of $i$-the arc. }
 \label{ithrow}
\end{figure}

Note that neither of the cases of $i=i-1$ nor $j=j-1$ occurs in the left and middle of Figure~\ref{ithrow} since $D$ is a flat virtual ``knot'' diagram, and that $i=i-1$ occurs in the right of Figure~\ref{ithrow} if and only if $k=1$. 
Let $A$ be the $k$-square matrix whose $i$-th row is the $i$-th row constructed above. 

\begin{claim} \label{kkquasideterminant}
The $(k,k)$-quasideterminant $|A|_{k,k}$ of $A$ can be defined, and its image under $\pi_{0,n}$ is $0$. 
\end{claim}

\begin{proof}
If $k=1$, then $A$ is the zero-matrix of size $1$. Thus the statement holds. Suppose that $k>1$. 
First, consider the matrix $A'=\pi_{0,n}(A)$, whose $(i,j)$-entry is $\pi_{0,n}\left((A)_{i,j}\right)$. 
It can be regarded as the result of substituting $0$ for $t$. 
$A'$ is a form of
$\begin{pmatrix}
1&0&0&\cdots&0&-s^{\epsilon}\\
-s^{\epsilon}&1&0&\cdots&0&0\\
0&-s^{\epsilon}&1&\cdots&0&0\\
\vdots& \vdots &\vdots  &\ddots&\vdots & \vdots \\
0&0&0&\cdots&1&0\\
0&0&0&\cdots &-s^{\epsilon}&1
\end{pmatrix}$, where $\epsilon$'s are in $\{-1,0,1\}$. 
We transform this matrix as follows: 
For each $i=1,\dots,k-1$ in this order, delete all entries of the $i$-th column of $A'$ but the $(i,i)$-entry using the $i$-th row of $A'$.
Note that at each steps of this deletion, all diagonals but the $(k,k)$-entry are unit elements of $\mathcal{S}_{0}=\mathbb{Z}[s^{\pm1}]$, that after this deletion, each $(i,k)$-entry is unit element of $\mathcal{S}_{0}=\mathbb{Z}[s^{\pm1}]$ and that after this deletion, the $(k,k)$-entry becomes $0$ since $D$ is a flat virtual ``knot'' diagram. 
The matrix after this deletion is a form of 
$\begin{pmatrix}
1&0&0&\cdots&0&-s^{n_1}\\
0&1&0&\cdots&0&-s^{n_2}\\
0&0&1&\cdots&0&-s^{n_3}\\
\vdots& \vdots &\vdots  &\ddots&\vdots & \vdots \\
0&0&0&\cdots&1&-s^{n_{k-1}}\\
0&0&0&\cdots &0&0
\end{pmatrix}$, where $n_{i}$'s are integers. 
Note also that during this deletion, the $k$-th row is not added to another row. 

Next, following the above deletion, we transform $A$. 
After deletion, the matrix is a form of 
$\begin{pmatrix}
1+\square t&\square t&\square t&\cdots&\square t&-s^{n_1}+\square t\\
\square t&1+\square t&\square t&\cdots&\square t&-s^{n_2}+\square t\\
\square t&\square t&1+\square t&\cdots&\square t&-s^{n_3}+\square t\\
\vdots& \vdots &\vdots  &\ddots&\vdots & \vdots \\
\square t&\square t&\square t&\cdots&1+\square t&-s^{n_{k-1}}+\square t\\
\square t&\square t&\square t&\cdots &\square t&\square t
\end{pmatrix}$, where $\square t$'s are $0$'s or some sums of monomials of $\mathcal{S}_{n}$ each of whose degree is grater than $0$.   
Note that all diagonals but the $(k,k)$-entry and the $(i,k)$-entry for $i\neq k$ are unit elements by Lemma~\ref{unit}. 
We multiple a unit element for each $i$-th row from left so that the $(i,i)$-entry becomes $1$ for $i=1,\dots,k-1$. 
Note that each $(i,k)$-entry remains to be a unit element for $i=1,\dots, k-1$, and the other entries are $0$'s or some sums of monomials of $\mathcal{S}_{n}$ each of whose degree is grater than $0$. 
After this multiplication, the matrix is a form of 
$\begin{pmatrix}
1&\square t&\square t&\cdots&\square t&-s^{n_1}+\square t\\
\square t&1&\square t&\cdots&\square t&-s^{n_2}+\square t\\
\square t&\square t&1&\cdots&\square t&-s^{n_3}+\square t\\
\vdots& \vdots &\vdots  &\ddots&\vdots & \vdots \\
\square t&\square t&\square t&\cdots&1&-s^{n_{k-1}}+\square t\\
\square t&\square t&\square t&\cdots &\square t&\square t
\end{pmatrix}$. 
Then we continue to delete. 
Using the first row, delete all entries of the first column of the matrix but the $(1,1)$-entry. 
Note that in this process, the first row is added to another row after multiplying some sum of monomials each of whose degree is grater than $0$ from left. 
This implies that after this process, all diagonals but the $(k,k)$-entry and $(i,k)$-entry for $i\neq k$ remain to be unit elements by Lemma~\ref{unit}, and that the $(k,k)$-entry remains to be $0$'s some sum of monomials of $\mathcal{S}_n$ each of whose degree is grater than $0$. After this, delete all entries of the second column of the matrix but the $(2,2)$-entry using the second row. 
Repeating this until we delete all entries of the $(k-1)$-th column of the matrix but the $(k-1,k-1)$-entry using the $(k-1)$-th row. 
The obtained matrix is a form of 
$\begin{pmatrix}
1&0&0&\cdots&0&u_{1}\\
0&1&0&\cdots&0&u_2\\
0&0&1&\cdots&0&u_3\\
\vdots& \vdots &\vdots  &\ddots&\vdots & \vdots \\
0&0&0&\cdots&1&u_{k-1}\\
0&0&0&\cdots &0&F
\end{pmatrix}$, where $u_{i}$'s are unit elements of $\mathcal{S}_n$ and $F$ is an element of $\mathcal{S}_{n}$ such that $\pi_{0,n}(F)=0$. 
Note that during this deletion, the $k$-th row is not added to another row. 
We see that the $(k,k)$-quasideterminant of this matrix is $F$. 
Since this matrix is obtained from $A$ by a sequence of operations of adding some row but the $k$-th row after multiplying an element of $\mathcal{S}_n$ from left to another row and multiplying some row but the $k$-th row by a unit element from left, we see that the $(k,k)$-quasideterminant of $A$ is also $F$ by Lemma~\ref{divide} and \ref{rowtransformation}. 
\end{proof}

\begin{defini}\label{def_sqncoloring}
We call this $|A|_{k,k}\in \mathcal{S}_{n}$ modulo multiplying unit element $U$ of $\mathcal{S}_{n}$ from the left and multiplying unit element $V$ of $\mathcal{S}_n$ from the right such that $\pi_{0,n}\left(UV \right)=1$ {\it the $\mathcal{S}_{n}$-coloring invariant} of flat virtual knot $\alpha$, and it is denoted by $F_{\alpha}$. 
\end{defini}

We should show that this is independent of (1) the choice of starting arc with fixed pair $(D,V)$ (Claim~\ref{independence_startingarc}), (2) the choice of the fake vertices $V_{fake}$ with fixed diagram $D$ (Claim~\ref{independence_fake}), and (3) the choice of the diagram $D$ representing $\alpha$ (Claim~\ref{independence_diagram}). 

\begin{rmk} \label{rmk_sqncoloring}
The above construction is interpreted through semiquandle coloring using $\mathcal{S}_n$. 
Fix $(D,V)$ and choose one arc. 
Assign variable $x_i$ to the arc labeled $i$ for each $i$. 
Assign equations for each vertex as in Figure~\ref{sqncoloring_vertex}. 
Order the equations from the $k$-th arc. 
Then the equations are represented as 
$A\begin{pmatrix}
x_{1}\\
\vdots \\
x_{k}
\end{pmatrix}=
\begin{pmatrix}
0\\
\vdots \\
0
\end{pmatrix}$. 
In Claim~\ref{kkquasideterminant}, we see the matrix $A$ is transformed into a form of 
$\begin{pmatrix}
1&0&0&\cdots&0&u_{1}\\
0&1&0&\cdots&0&u_2\\
0&0&1&\cdots&0&u_3\\
\vdots& \vdots &\vdots  &\ddots&\vdots & \vdots \\
0&0&0&\cdots&1&u_{k-1}\\
0&0&0&\cdots &0&F
\end{pmatrix}$ 
without summing the $k$-th row to another row and without multiplying the $k$-th row by a unit element from left. 
This can be interpreted as follows: 
Removing a point in the $k$-th arc, cut the arc into two arcs, called {\it the starting arc} and {\it the terminal arc}. 
We assign the starting arc the variable $x_0$ and continue to assign the terminal arc the variable $x_k$. 
Then the equation of the cut diagram is represented as 
$\overline{A}\begin{pmatrix}
x_{1}\\
\vdots \\
x_{k-1}\\
x_0\\
x_{k}
\end{pmatrix}=
\begin{pmatrix}
0\\
\vdots \\
0\\
0\\0
\end{pmatrix}$, where $\overline{A}$ is a $k\times(k+1)$-matrix obtained from $A$ by adding the zero-column vector with $k$ entries to the right of $A$, and then exchange the $(k,k)$-entry and the $(k,k+1)$-entry. 
By following the transformation of $A$, this equations are transformed to 
$\begin{pmatrix}
1&0&0&\cdots&0&u_{1}&0\\
0&1&0&\cdots&0&u_2&0\\
0&0&1&\cdots&0&u_3&0\\
\vdots& \vdots &\vdots  &\ddots&\vdots & \vdots &\vdots \\
0&0&0&\cdots&1&u_{k-1}&0\\
0&0&0&\cdots &0&-1+F&1
\end{pmatrix}
\begin{pmatrix}
x_{1}\\
\vdots \\
x_{k-1}\\
x_0\\
x_{k}
\end{pmatrix}=
\begin{pmatrix}
0\\
\vdots \\
0\\
0\\0
\end{pmatrix}$. 
This means that the coloring of this cut diagram is determined by the coloring of the starting arc. 
Under this view, the element $F$ results from the equation $x_0=x_k$. This equation is necessary for the coloring of the cut diagram to be a coloring of the original diagram. 
\end{rmk}

\begin{figure}[htbp]
 \begin{center}
  \includegraphics[width=70mm]{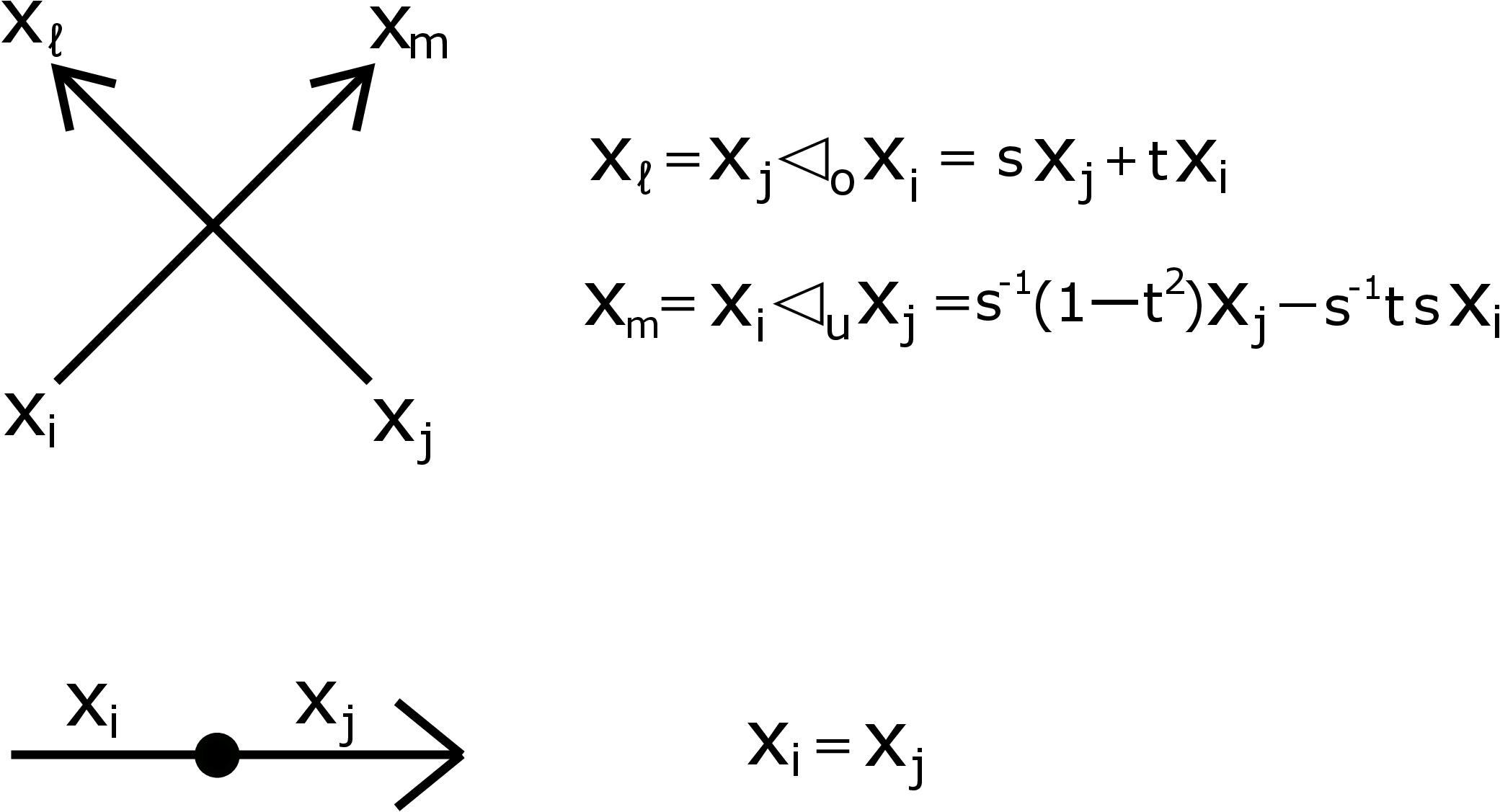}
 \end{center}
 \caption{Equations at flat vertex (top) and fake vertex (bottom) }
 \label{sqncoloring_vertex}
\end{figure}

\begin{claim}\label{independence_startingarc}
For a fixed $(D,V)$, the $\mathcal{S}_{n}$-coloring invariant is independent of the choice of the starting arc. 
\end{claim}

\begin{proof}
Choose a starting arc and give the arcs labels $1,\dots,k$ from it. 
Let $A$ be the $k$-square matrix constructed by the above operation. 
Moreover, let $B$ be the $k$-square matrix constructed by choosing the $(k-1)$- labeled arc as the starting arc. 
We will show $|A|_{k,k}$ is same as $|B|_{k,k}$ modulo multiplying unit elements $U$ and $V$ of $\mathcal{S}_n$ with $\pi_{0,n}(UV)=1$ from left and right, respectively. 
Note that $B$ is transformed into $A$ by s sequence of operations: Exchanging the first row and the second row, exchanging the second row and the third row,..., exchanging the $(k-1)$-th row and the $k$-th row, and then exchanging the first column and the second column, exchanging the second column and the third column,..., exchanging the $(k-1)$-th column and the $k$-th column. 
Thus by Lemma~\ref{permutation}, we see that $|B|_{k,k}=|A|_{k-1,k-1}$. 
Following the transformation of $A$ in Claim~\ref{kkquasideterminant}, being careful not to add the $(k-1)$-th row to another rows and not to multiple $(k-1)$-th row by a unit element from left, we get a matrix 
$\tilde{A}=\begin{pmatrix}
1&0&\cdots&0&*&*\\
0&1&\cdots&0&*&*\\
\vdots& \vdots  &\ddots&\vdots&\vdots & \vdots \\
0&0&\cdots&1&*&*\\
0&0&\cdots&0&a&b\\
0&0&\cdots&0 &c&d
\end{pmatrix}$ for some elements $a,b,c,d$ of $\mathcal{S}_n$. 
Note that we see that $\pi_{0,n}(a)=\pi_{0,n}(d)=1$ by following the deletion of $\pi_{0,n}\left( A \right)$. 
Since neither the $(k-1)$-th row nor the $k$-th row is added to another row during this transformation, we see that $|\tilde{A}|_{k,k}=|A|_{k,k}$ and $|\tilde{A}|_{k-1,k-1}=|A|_{k-1,k-1}$ by Lemma~\ref{rowtransformation}. 
By Lemma~\ref{unit}, we see that $a$ and $d$ are unit elements of $\mathcal{S}_n$. 
Moreover, we have $\pi_{0,n}(a) \pi_{0,n}(d)=\pi_{0,n}(b)\pi_{0,n}(c)$ since the ordinary determinant of $\pi_{0,n}\left( A\right)$ (and thus that of $\pi_{0,n}\left( \tilde{A} \right)$) are $0$. 
This implies that $\pi_{0,n}(b)=s^{l}$, $\pi_{0,n}(c)=s^{-l}$ for some integer $l$ and that $b$ and $c$ are unit elements of $\mathcal{S}_n$ by Lemma~\ref{unit}. 
By using the $(k-1)$-th row, we can delete the $(k,k-1)$-entry of $\tilde{A}$. 
Since the $(k,k)$-entry of the resultant matrix is $|A|_{k,k}$, we have $d-ca^{-1}b=|A|_{k,k}$.  
By this we have $c^{-1}db^{-1}=a^{-1}+c^{-1}|A|_{k,k}b^{-1}=\left( 1+c^{-1}|A|_{k,k}b^{-1}a\right)a^{-1}$. 
Thus we have $bd^{-1}c=\left( \left( 1+c^{-1}|A|_{k,k}b^{-1}a\right)a^{-1}\right)^{-1}=a{\displaystyle \sum^{n}_{i=0} (-1)^{i} \left(c^{-1}|A|_{k,k}b^{-1}a \right)^{i}}$. 
Note that $c^{-1}|A|_{k,k}b^{-1}a $ is $0$ or a some sum of monomials each of whose degree is grater than $0$. 
By computation, we get 
\begin{align}
|A|_{k-1,k-1}&=|\tilde{A}|_{k-1,k-1}\notag \\
&=a-
\begin{pmatrix}
0&\cdots&0&b
\end{pmatrix}
{\begin{pmatrix}
1&0&\cdots&0&*\\
0&1&\cdots&0&*\\
\vdots& \vdots  &\ddots&\vdots& \vdots \\
0&0&\cdots&1&*\\
0&0&\cdots&0 &d
\end{pmatrix}}^{-1}
\begin{pmatrix}
*\\
 \vdots \\
*\\
c
\end{pmatrix}=a-bd^{-1}c=a{\displaystyle \sum^{n}_{i=1} (-1)^{i+1} \left(c^{-1}|A|_{k,k}b^{-1}a \right)^{i}}\notag \\
&=a\left(c^{-1}|A|_{k,k}b^{-1}a \right){\displaystyle \sum^{n}_{i=0} (-1)^{i} \left(c^{-1}|A|_{k,k}b^{-1}a \right)^{i}}=\left(ac^{-1}\right)|A|_{k,k}\left(b^{-1}a{\displaystyle \sum^{n}_{i=0} (-1)^{i} \left(c^{-1}|A|_{k,k}b^{-1}a \right)^{i}}\right)
. \notag
\end{align} 
By applying $\pi_{0,n}$ we see that $\left(ac^{-1}\right)$ and $\left(b^{-1}a{\displaystyle \sum^{n}_{i=0} (-1)^{i} \left(c^{-1}|A|_{k,k}b^{-1}a \right)^{i}}\right)$ are unit elements of $\mathcal{S}_{n}$ by Lemma~\ref{unit}. 
Moreover, the product of the images of these elements under the projection $\pi_{0,n}$ is $1$. 
\end{proof}

\begin{claim}\label{independence_fake}
For a fixed diagram $D$ of a flat virtual knot $\alpha$, the $\mathcal{S}_{n}$-coloring invariant $F_{\alpha}$ is independent of choices of fake vertices $V_{fake}\subset V$. 
\end{claim}

\begin{proof}
We consider $F_{\alpha}$ through a system of equations as in Remark~\ref{rmk_sqncoloring}. 
Take fake vertices $V_{fake}$ for a diagram $D$. 
Consider new fake vertices $V'_{fake}=V_{fake}\cup \{v\}$, obtained by adding new fake vertex $v$ to $V_{fake}$. 
Set $V'$ to be $V_{flat}\cup V'_{fake}$
Take a point $p$ on $D$ which is not a point of $V'_{fake}$. 
On both of $(D,V)$ and $(D,V')$, choose the arcs containing the point $p$ as a starting arc. 
Then after substituting, the system of equations on $(D,V')$ is obtained from that on $(D,V)$ by adding new variable with an equation explaining it. See Figure~\ref{indep_fake}. 
In order to substitute, we change the starting arc so that the last equation does not come from the added fake vertex $v$. 
This does not affect the computation of $F_{\alpha}$ by Claim~\ref{independence_startingarc}. 
\end{proof}

\begin{figure}[htbp]
 \begin{center}
  \includegraphics[width=100mm]{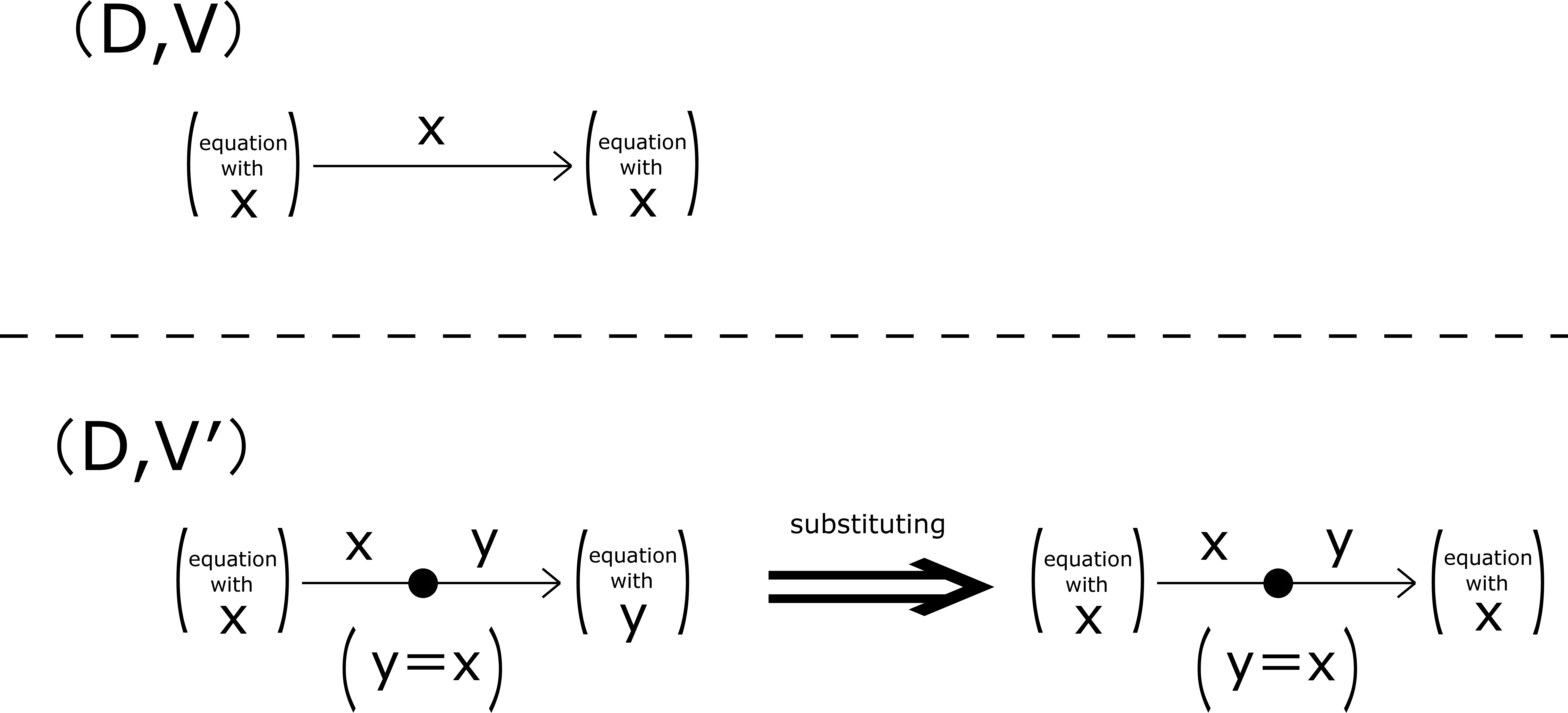}
 \end{center}
 \caption{Top: A system of equations on $(D,V)$ \ \ \ \ \ Bottom: A system of equations on $(D,V')$ }
 \label{indep_fake}
\end{figure}

\begin{claim}\label{independence_diagram}
The $\mathcal{S}_{n}$-coloring invariant $F_{\alpha}$ for a flat virtual knot $\alpha$ is independent of flat virtual knot diagrams representing $\alpha$. 
\end{claim}

\begin{proof}
Let $D$ and $D'$ be two flat virtual knot diagrams representing $\alpha$ such that they are related by one move in Figure~\ref{fivemoves}. 
Take same fake vertices for $D$ and $D'$ outside of the areas depicted in Figure~\ref{fivemoves} so that the endarcs of each figure in Figure~\ref{fivemoves} are distinct. 
Take same starting arcs for $D$ and $D'$ outside of the areas depicted in Figure~\ref{fivemoves}. 
Then the fact that multiple element at the output coloring with respect to the input coloring does not change follows from Remark~\ref{invariance_coloring}. 
\end{proof}

\begin{example}
We compute (one of representatives of) the $\mathcal{S}_n$-coloring invariant of $3.1$. 
We choose a diagram of $3.1$ and an arc of it, and assign variables as in Figure~\ref{example_computation}. 

\begin{figure}[htbp]
 \begin{center}
  \includegraphics[width=30mm]{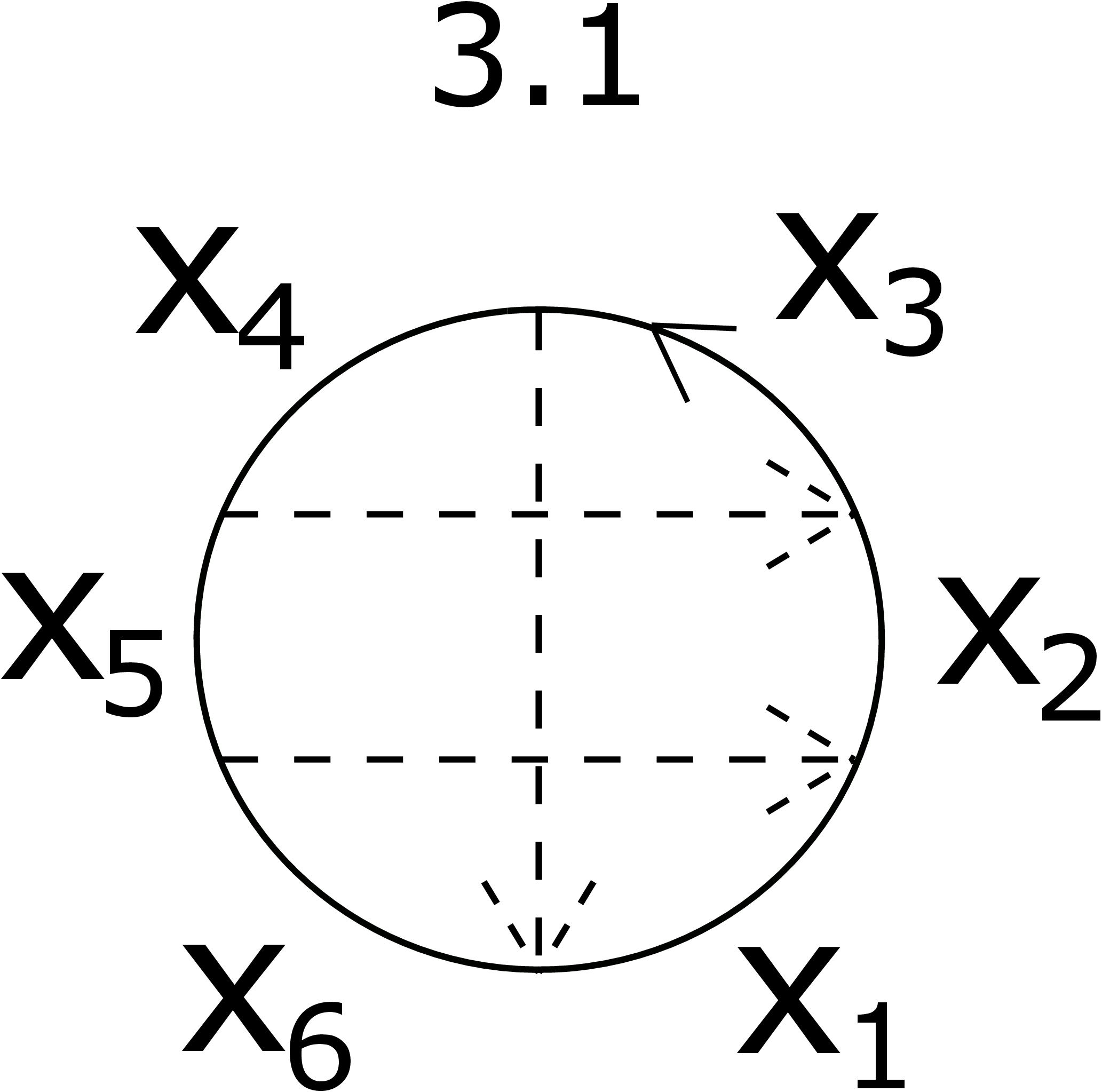}
 \end{center}
 \caption{A Gauss diagram of $3.1$ and an assignment of variables.}
 \label{example_computation}
\end{figure}

Then the matrix is transformed as follows:
\begin{align}
&\begin{pmatrix}
1&0&-t&0&0&-s \\
-s&1&0&0&-t&0 \\
0&-s&1&-t&0&0 \\
0&0&-s^{-1}(1-t^2)&1&0&s^{-1}ts \\
0&s^{-1}ts&0&-s^{-1}(1-t^2)&1&0 \\
s^{-1}ts&0&0&0&-s^{-1}(1-t^2)&1
\end{pmatrix} \notag \\
& \ \downarrow \notag \\
&
\begin{pmatrix}
1&0&-t&0&0&-s \\
0&1&-st&0&-t&-s^{2} \\
0&-s&1&-t&0&0 \\
0&0&-s^{-1}(1-t^2)&1&0&s^{-1}ts \\
0&s^{-1}ts&0&-s^{-1}(1-t^2)&1&0 \\
0&0&s^{-1}tst&0&-s^{-1}(1-t^2)&1+s^{-1}ts^2
\end{pmatrix} \notag  \\
& \ \downarrow \notag \\
&
\begin{pmatrix}
1&0&-t&0&0&-s \\
0&1&-st&0&-t&-s^{2} \\
0&0&1-s^{2}t&-t&-st&-s^3 \\
0&0&-s^{-1}(1-t^2)&1&0&s^{-1}ts \\
0&0&s^{-1}ts^{2}t&-s^{-1}(1-t^2)&1+s^{-1}tst&s^{-1}ts^3 \\
0&0&s^{-1}tst&0&-s^{-1}(1-t^2)&1+s^{-1}ts^2
\end{pmatrix} \notag \\
& \ \downarrow \notag \\
&
\begin{pmatrix}
1&0&-t&0&0&-s \\
0&1&-st&0&-t&-s^{2} \\
0&0&1-s^{2}t&-t&-st&-s^3 \\
0&0&0&1-s^{-1}(1-t^2)(1-s^{2}t)^{-1}t&-s^{-1}(1-t^2)(1-s^{2}t)^{-1}st&s^{-1}\left( t-(1-t^2)(1-s^{2}t)^{-1}s^2\right)s \\
0&0&0&s^{-1}\left(ts^{2}t (1-s^{2}t)^{-1}t-(1-t^2)\right)&1+s^{-1}t \left(1+s^{2}t(1-s^{2}t)^{-1} \right)st&s^{-1}t\left(1+ s^{2}t(1-s^{2}t)^{-1} \right)s^3 \\
0&0&0&s^{-1}tst(1-s^{2}t)^{-1}t&s^{-1}\left( tst(1-s^{2}t)^{-1}st-(1-t^2)\right)&1+s^{-1}ts\left( 1+t(1-s^{2}t)^{-1}s^{2}\right)s
\end{pmatrix} \notag \\
& \ \downarrow \notag \\
&
\begin{pmatrix}
1&0&-t&0&0&-s \\
0&1&-st&0&-t&-s^{2} \\
0&0&1-s^{2}t&-t&-st&-s^3 \\
0&0&0&a&b&c \\
0&0&0&d&e&f \\
0&0&0&g&h&i
\end{pmatrix} \notag \\
& \ \downarrow \notag \\
&
\begin{pmatrix}
1&0&-t&0&0&-s \\
0&1&-st&0&-t&-s^{2} \\
0&0&1-s^{2}t&-t&-st&-s^3 \\
0&0&0&a&b&c \\
0&0&0&0&e-da^{-1}b&f-da^{-1}c \\
0&0&0&0&h-ga^{-1}b&i-ga^{-1}c
\end{pmatrix} \notag \\
& \ \downarrow \notag \\
&
\begin{pmatrix}
1&0&-t&0&0&-s \\
0&1&-st&0&-t&-s^{2} \\
0&0&1-s^{2}t&-t&-st&-s^3 \\
0&0&0&a&b&c \\
0&0&0&0&e-da^{-1}b&f-da^{-1}c \\
0&0&0&0&0&i-ga^{-1}c-(h-ga^{-1}b)(e-da^{-1}b)^{-1}(f-da^{-1}c)
\end{pmatrix} \notag 
\end{align}

Thus we compute $F_{3.1}=i-ga^{-1}c-(h-ga^{-1}b)(e-da^{-1}b)^{-1}(f-da^{-1}c)$, where $a=1-s^{-1}(1-t^2)(1-s^{2}t)^{-1}t$, $b=-s^{-1}(1-t^2)(1-s^{2}t)^{-1}st$, $c=s^{-1}\left( t-(1-t^2)(1-s^{2}t)^{-1}s^2\right)s$, $d=s^{-1}\left(ts^{2}t (1-s^{2}t)^{-1}t-(1-t^2)\right)$, $e=1+s^{-1}t \left(1+s^{2}t(1-s^{2}t)^{-1} \right)st$, $f=s^{-1}t\left(1+ s^{2}t(1-s^{2}t)^{-1} \right)s^3$, $g=s^{-1}tst(1-s^{2}t)^{-1}t$, $h=s^{-1}\left( tst(1-s^{2}t)^{-1}st-(1-t^2)\right)$ and $i=1+s^{-1}ts\left( 1+t(1-s^{2}t)^{-1}s^{2}\right)s$.

\end{example}

\subsection{Some properties}
In this subsection, we collect some properties for the $\mathcal{S}_{n}$-coloring invariant for a flat virtual knot. 

\begin{defini}
For a flat virtual knot $\alpha$, take some representative $F_{\alpha}\in \mathcal{S}_{n}$ of the $\mathcal{S}_n$-coloring invariant of $\alpha$. 
Using the basis as in Proposition~\ref{basis}, this $F_{\alpha}$ is uniquely represented as $F_{\alpha}={\displaystyle \sum^{n}_{i=1} f_{\alpha,i}(s) t^{i}}$ for some Laurent polynomials $f_{\alpha,i}$ for $i=1,\dots ,n$. 
After taking identification caused by taking modulo on $F$, we call the image of $f_{\alpha,i}(s)$ {\it the $i$-th $\mathcal{S}_{n}$-coloring polynomial} of $\alpha$.   
\end{defini} 

\subsection*{Observations}
\begin{itemize}
\item Let $\alpha$ be a flat virtual knot. Let $n$ and $m$ be non-negative integers with $n\leq m$. Let $F$ and $G$ be the $\mathcal{S}_{n}$-coloring invariant and $\mathcal{S}_{m}$-coloring invariant of $\alpha$. 
Then $F=\pi_{m,n}\left( G\right)$ holds since semiquandle operations, thus constructions of equations are compatible with the projection $\pi_{m,n}$. 
\item Let ${\displaystyle \sum^{n}_{i=1} f_{\alpha,i}(s) t^{i}}$ be a representative of the $\mathcal{S}_{n}$-coloring invariant of a flat virtual knot $\alpha$. 
Then the first $\mathcal{S}_n$-coloring polynomial is defined in $\mathbb{Z}[s^{\pm1}]$. 
Moreover, for $i>1$, suppose that the $f_{\alpha,j}$ is $0$ for all $j<i$. Then the $i$-th $\mathcal{S}_{n}$-coloring polynomial is defined in $\mathbb{Z}[s^{\pm1}]$. 
This is because the ambiguity for defining the $\mathcal{S}_{n}$-coloring invariant is multiplying a unit element $U$ from left and multiplying a unit element $V$ from right such that $\pi_{0,n}(UV)=1$ as stated in Definition~\ref{def_sqncoloring}. 
\item By the above, for computing the $i$-th $\mathcal{S}_n$-coloring polynomial, we can use any non-negative integer $m$ with $i\leq m$ and can compute the $i$-th $\mathcal{S}_{m}$-coloring polynomial instead of computing the $i$-th $\mathcal{S}_n$-coloring polynomial. 
\end{itemize}

In general, it is complicated to compute the $i$-th $\mathcal{S}_{n}$-coloring polynomial. 
However, the first $\mathcal{S}_{n}$-coloring polynomial relates to the u-polynomial, defined in \ref{u-polynomial} in Section~\ref{sec_flatvirtuallinks}. 

\begin{prop}
Let $\alpha$ be a flat virtual knot. Then $f_{\alpha,1}(s)=u_{\alpha}(s)-u_{\alpha}(s^{-1})$, where $f_{\alpha,1}(s)$ is the first $\mathcal{S}_{n}-$coloring polynomial of $\alpha$ and $u_{\alpha}(s)$ is the u-polynomial of $\alpha$. 
\end{prop}

\begin{proof}
We use $\mathcal{S}_1$-coloring. Note that $t^{2}=0$ holds in $\mathcal{S}_1$ and thus $\mathcal{S}_1$ is commutative. 
We will work on a Gauss diagram $D$ of $\alpha$. 
If an arc colored by $x$ passes the head of the dashed arrow $e$, then the next arc should be colored by $(s+s^{1+n(e)}t)x$, see the top of Figure~\ref{firstpoly}. 
And if an arc colored by $x$ passes the tail of the dashed arrow $e$, then the next arc should be colored by $(s^{-1}-s^{-1-n(e)}t)x$, see the bottom of Figure~\ref{firstpoly}. 
Take any point on the circle of the Gauss diagram and give the arc in which the point is the variable $x$. 
By traveling  the circle, we have the coloring ${\displaystyle \prod_{e\in arr(D)} \left(s+s^{1+n(e)}t \right) \left( s^{-1}-s^{-1-n(e)}t\right)}={\displaystyle \prod_{e\in arr(D)} \left( 1+\left( s^{n(e)}-s^{-n(e)}\right)t \right) }=1+ t\left( {\displaystyle \sum_{e\in arr(D)} s^{n(e)}-s^{-n(e)} }\right)$. 
By this, we see that $f_{\alpha,1}(s)=\left( {\displaystyle \sum_{e\in arr(D)} s^{n(e)}-s^{-n(e)} }\right)=\left( {\displaystyle \sum_{\substack{e\in arr(D) \\ n(e)=0}}s^{n(e)}-s^{-n(e)} }\right) + \left( {\displaystyle \sum_{\substack{e\in arr(D) \\ n(e)>0}}s^{n(e)}-s^{-n(e)} }\right) + \left( {\displaystyle \sum_{\substack{e\in arr(D) \\ n(e)<0}}s^{n(e)}-s^{-n(e)} }\right) = \left({\displaystyle \sum_{e\in arr(D)}sign(n(e))s^{|n(e)|}}\right) - \left({\displaystyle \sum_{e\in arr(D)}sign(n(e))s^{-|n(e)|}}\right) =u_{\alpha}(s)-u_{\alpha}(s^{-1})$. 
\end{proof}

\begin{figure}[htbp]
 \begin{center}
  \includegraphics[width=80mm]{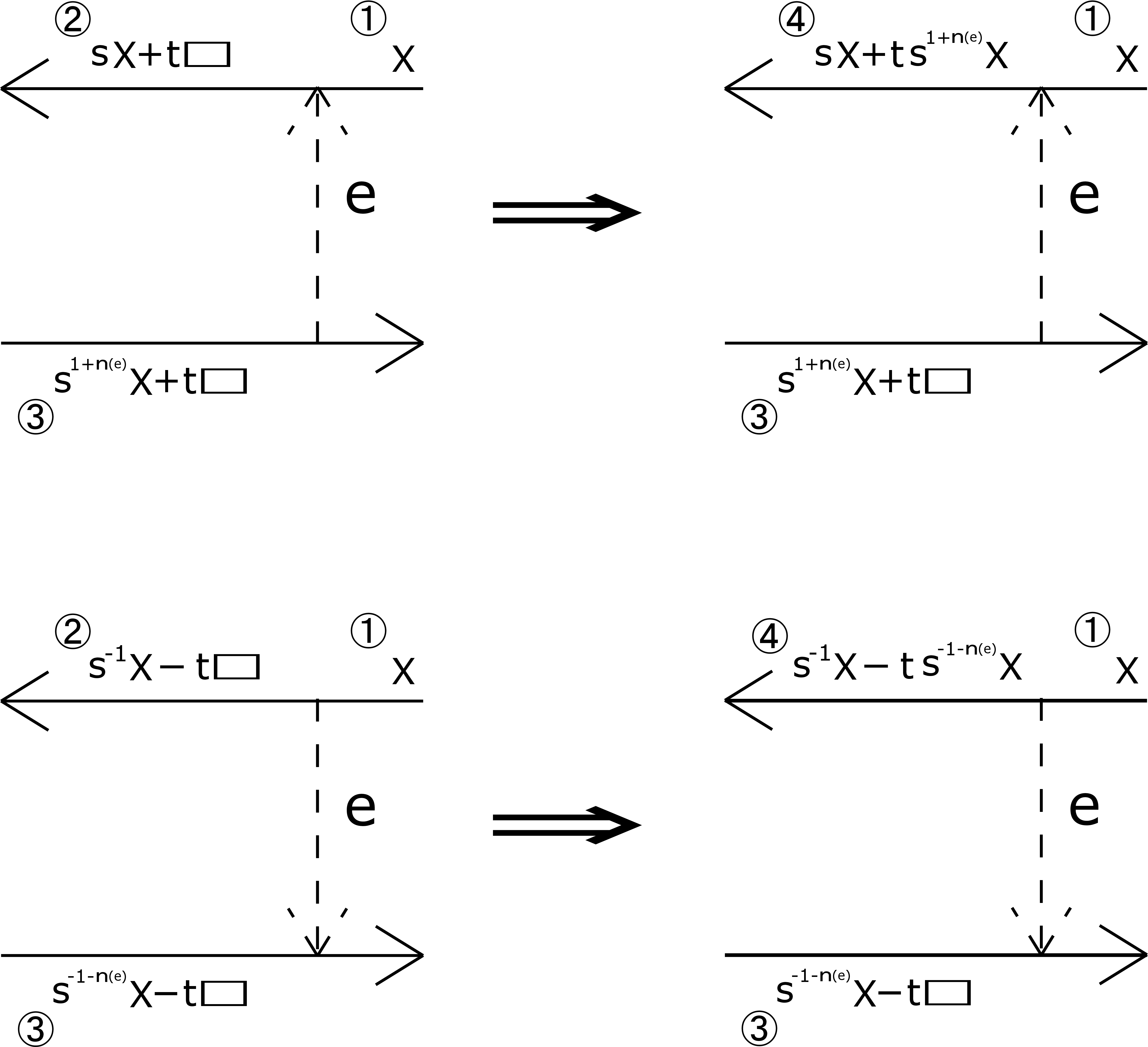}
 \end{center}
 \caption{Changes of colorings of arcs through passing crossings. }
 \label{firstpoly}
\end{figure}

By the above, we can compute the first $\mathcal{S}_{n}$-coloring polynomial using the u-polynomial. 
Conversely, we can reconstruct the u-polynomial using the first $\mathcal{S}_{n}$-coloring polynomial by discarding the terms of negative power. Note that the u-polynomials (and thus the first $\mathcal{S}_{n}$-coloring polynomial) have no constant terms. 
Therefore, the first $\mathcal{S}_{n}$-coloring polynomial has as much information as the u-polynomial has for classifying flat virtual knots. 
By direct computation, we have the following: 

\begin{prop}
The pair of the first and second $\mathcal{S}_{n}$-coloring polynomials has more information than the u-polynomial. 
\end{prop}

\begin{proof}
The u-polynomial of the flat virtual knot $4.11$ is $0$, and thus it cannot be distinguished from the trivial flat virtual knot by the u-polynomials. 
The first $\mathcal{S}_{n}$-coloring polynomials of them are $0$, and the second $\mathcal{S}_{n}$-coloring polynomials of trivial flat virtual knot and $4.11$ are $0$ and $s^{2}-2s-1+2s^{-1}-s^{-2}$, respectively by computation. 
\end{proof}

\begin{prop}
There exists a non-trivial flat virtual knot whose $\mathcal{S}_{n}$-coloring invariant is $0$ for all $n$. 
\end{prop}

\begin{proof}
The $\mathcal{S}_{n}$-coloring invariant of the trivial flat virtual knot is $0$. 
Thus under the interpretation in Remark~\ref{rmk_sqncoloring}, we see the following: 
For every flat virtual knot diagram $D$ of the trivial flat virtual knot and every point $p$ on any arc of $D$, cut $D$ at $p$ and solve the obtained equation. 
Then the input coloring and output coloring are identical. 
This means that we cannot detect trivial flat virtual knot summands of composite flat virtual knots by $\mathcal{S}_{n}$-coloring invariant. 
Especially, the $\mathcal{S}_{n}$-coloring invariant of the flat virtual knot $4.10$ in Figure~\ref{example_flatvirtual} or \ref{example_gauss}, which is the resultant of the connected sum of two diagrams of trivial flat virtual knots is $0$. 
\end{proof}

\ \ E-mail address: \texttt{nnsekino@gmail.com}


\begin{thebibliography}{99}
\bibitem{chen1} Jie Chen, Flat knots and Invariants, 2023. Ph.D. Thesis, McMaster University. 

\bibitem{chen2} Jie Chen, Connected sum and crossing numbers of flat virtual knots, arXiv: 2312.03994.

\bibitem{flatknotinfo} Jie Chen, FlatKnotInfo: Table of Flat Knot Invariants, https://flatknotinfo.mcmaster.ca.

\bibitem{gibson1} A. Gibson, Homotopy invariants of Gauss phrases, Indiana Univ. Math. J. 59 (1) (2010) 207--229. 

\bibitem{gibson2} A. Gibson, Homotopy invariants of Gauss words, Math. Ann. 349 (4) (2011) 871--887. 

\bibitem{henrich} Allison Henrich and Sam Nelson, Semiquandles and flat virtual knots, Pacific Journal of Mathematics. Vol. 248. No. 1, 2010. 

\bibitem{hrencecin} D. Hrencecin and L. Kauffman, On filamentations and virtual knot, Topol. Appl. 134 (1) (2003) 23-52. 

\bibitem{im} Y. H. Im, K. Lee and H. Son, An index polynomial invariant for flat virtual knots, Eur. J. Comb. 31 (8) (2010) 2130--2140. 

\bibitem{quasideterminant} Israel Moiseevich Gel'fand and Vladimir Retakh, Determinants of matrices over noncommutative rings, Functional Analysis and Its Applications. 25 (2): 91--102. doi:10.1007/BF01079588

\bibitem{semiquandle} Allison Henrich and Sam Nelson, Semiquandles and flat virtual knots, Pacific J. Math. 248 (2010), no. 1, 155--170, DOI 10.2140/pjm.2010.248.155. MR638121

\bibitem{kauffman} Louis H. Kauffman, Virtual knot theory, European J. Combin. 20 (1999), no. 7, 663--690.

\bibitem{kauffman2} L. Kauffman and B. Richter, A polynomial invariant for virtual links, J. Knot Theory Ramif. 17 (5) (2008) 521--528. 

\bibitem{lee} Kyeonghui Lee, Young Ho Im and Sera Kim, A family of polynomial invariants for flat virtual knots, Topology and its Applications 264  (2019) 413--419.  

\bibitem{minimalgenerating} Michael Polyak, Minimal generating sets of Reidemeister moves, Quantum Topol. 1 (2010), 399--411, DOI 10.4171/QT/10

\bibitem{upolynomial} Vladimir Turaev, Virtual strings, Ann. Inst. Fourier (Grenoble) 54 (2004), no. 7, 2455--2525

\end{thebibliography}
\end{document}